\documentclass[10pt,journal,compsoc]{IEEEtran}
\usepackage{xcolor}
\usepackage{amsmath,amsfonts}
\usepackage{array}
\usepackage{textcomp}
\usepackage{stfloats}
\usepackage{url}
\usepackage{verbatim}
\usepackage{graphicx}
\usepackage{capt-of}
\usepackage{hyperref}
\usepackage{algorithm}
\usepackage{algpseudocode}
\usepackage{amssymb}
\usepackage{placeins}
\usepackage{mathtools}
\usepackage{booktabs}
\usepackage{tabularx}
\usepackage{makecell}
\usepackage{enumitem}

\usepackage{amsthm}
\newtheorem{theorem}{Theorem}
\newtheorem{assumption}{Assumption}
\newtheorem{lemma}{Lemma}
\newtheorem{proposition}{Proposition}
\newtheorem{corollary}{Corollary}

\theoremstyle{definition}

\theoremstyle{remark}
\newtheorem{remark}{Remark}

\ifCLASSOPTIONcompsoc
  \usepackage[nocompress]{cite}
\else
  \usepackage{cite}
\fi
\ifCLASSOPTIONcompsoc
 \usepackage[caption=false,font=footnotesize,labelfont=sf,textfont=sf]{subfig}
\else
 \usepackage[caption=false,font=footnotesize]{subfig}
\fi

\begin{document}

\title{Spectral Sensitivity of Directed Weighted Networks: Why Weakening Edges May Trigger Synchronization
\thanks{
This work has been submitted to the IEEE for possible publication. Copyright may be transferred without notice, after which this version may no longer be accessible.
}
}

\author{Xinyu~Wu, 
Xizhi~Liu, 
Chenyao~Zhang, 
Tianping~Chen, 
and~Wenlian~Lu%
\IEEEcompsocitemizethanks{
\IEEEcompsocthanksitem Xinyu Wu and Xizhi Liu are with the School of Mathematical Sciences, University of Science and Technology of China, Hefei, 230026, China.
\IEEEcompsocthanksitem Chenyao Zhang, Tianping Chen and Wenlian Lu are with the School of Mathematical Sciences, Fudan University, Shanghai 200433, China.
\IEEEcompsocthanksitem Corresponding authors: Wenlian Lu and Xizhi Liu (email: wenlian@fudan.edu.cn; liuxizhi@ustc.edu.cn).
}%
}

\IEEEtitleabstractindextext{%
\begin{abstract}
Synchronization in dynamical systems on directed weighted networks is often associated with stronger coupling and denser interactions. This paper shows that the opposite can also occur: weakening selected edges may increase the generalized algebraic connectivity, denoted by $\kappa$, and in some nonlinear systems this spectral improvement is accompanied by a transition from nonsynchronization to synchronization. To explain this effect, we develop a perturbation-based spectral sensitivity framework for directed weighted networks. We derive an explicit first-order formula for the response of $\kappa$ to edge-weight perturbations and show that it decomposes into a directed cut-energy term and a stationary redistribution term. This decomposition clarifies how asymmetric flow structure and invariant-mass redistribution jointly determine the synchronization role of each edge. Based on this theory, we design sensitivity-guided algorithms for edge weakening, edge deletion, negative-edge insertion, and edge strengthening. Experiments on synthetic and real networks show that these methods identify critical edges whose modification yields substantial gains in $\kappa$. Simulations of first- and second-order nonlinear consensus dynamics further show markedly faster convergence and, in some cases, a transition from incoherence to synchronization. The results provide a local spectral mechanism by which reducing or reallocating coupling can enhance synchronization-related performance. 
\end{abstract}

}

\maketitle

\IEEEdisplaynontitleabstractindextext

%
\IEEEpeerreviewmaketitle

\IEEEraisesectionheading{\section{Introduction}\label{sec:introduction}}

\begin{figure}[!t]
    \centering
        \centering
        \includegraphics[width=\linewidth]{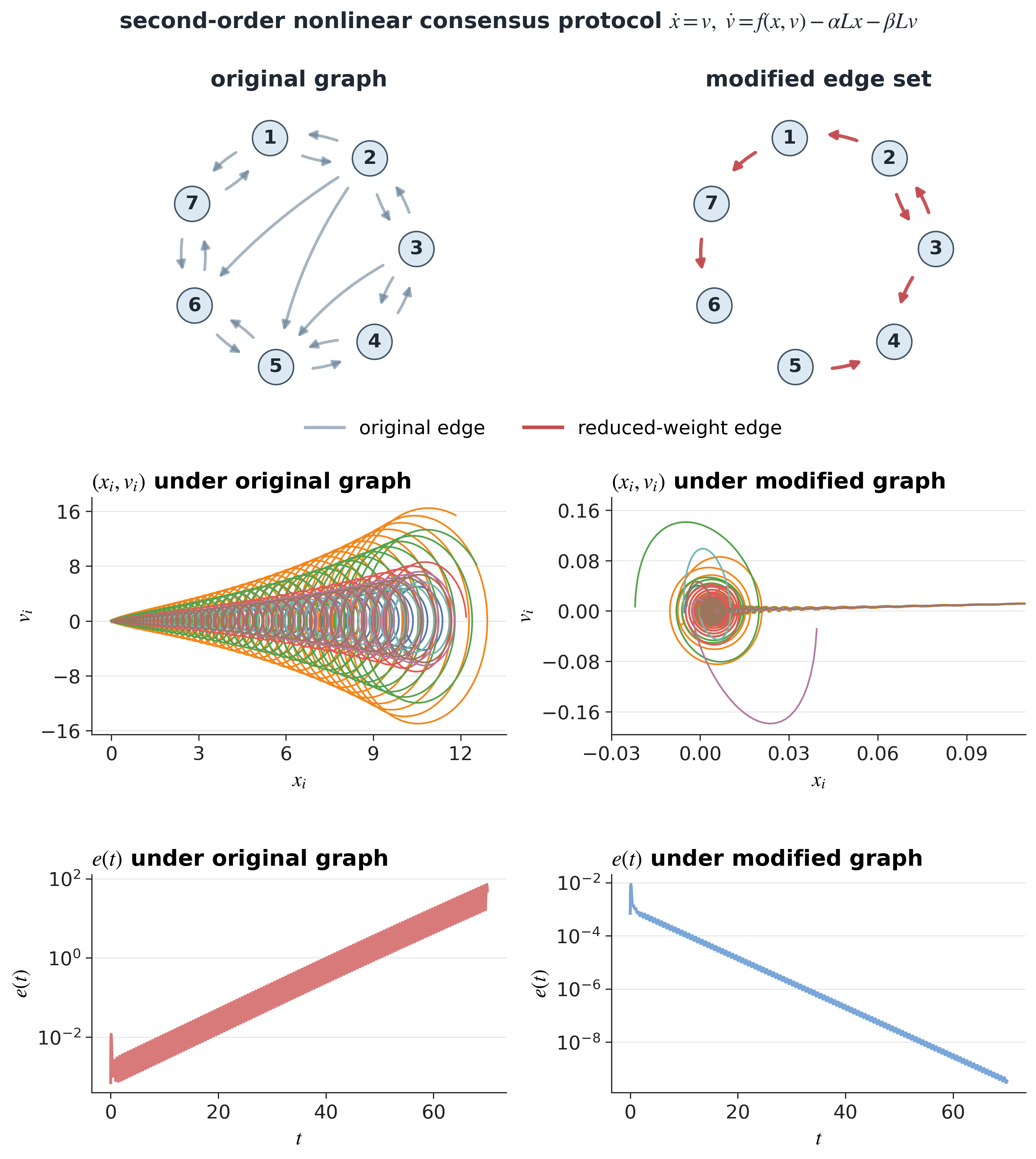}
        \caption{
Synchronization enhancement in a second-order nonlinear consensus system via edge-weight reduction. 
Top: original graph and modified parts with reduced weights (red). 
Middle: phase trajectories $(x_i(t), v_i(t))$. 
Bottom: mean-square error $e(t)$. 
}
        \label{fig:intro_second}
\end{figure}

\IEEEPARstart{S}{ynchronization}, in which dynamical units evolve toward a common behavior through network interactions, is a fundamental phenomenon observed in a wide range of natural and engineered systems, including neuronal networks~\cite{ashwin2016mathematical}, social interactions~\cite{proskurnikov2018tutorial}, and multi-agent coordination~\cite{Li2017Cooperative}. 
In many real-world scenarios, these interactions are inherently directional~\cite{Newman2018Networks} and are more appropriately modeled by directed graphs, where coupling is asymmetric and may even involve signed weights~\cite{Meng2022Disagreement} to encode both cooperative and antagonistic effects.

A prevailing intuition is that synchronization should generally benefit from stronger or more abundant interactions. 
In particular, increasing edge weights or adding connections is often expected to improve network connectivity and thereby facilitate collective behavior. 
This intuition is well aligned with classical spectral graph theory for undirected networks, where algebraic connectivity increases with coupling strength and plays a central role in characterizing convergence and synchronizability~\cite{fiedler1973algebraic,GodsilRoyle2001Algebraic,olfati2007consensus}.

However, such intuition does not extend straightforwardly to directed or signed networks. 
A growing body of work has revealed that, in asymmetric networked systems, structural modifications that appear locally detrimental, including link removal, weight reduction, and even the introduction of negative interactions, may in fact improve synchronization-related performance~\cite{nishikawa2010network,zeng2012manipulating,zhang2017effect,he2019perturbation,gao2023effects,bakrani2024cycle}. 
In particular, it has been shown that decreasing edge weights can accelerate convergence in certain directed settings; for example, \cite{gao2023effects2} demonstrates that, in leader--follower networks, properly reducing the weights of selected edges can improve the consensus convergence rate. 
These observations strongly suggest the existence of compensatory structural effects, whereby a local weakening of interaction may produce a global gain in dynamics. 

Despite these intriguing results, a unified theoretical framework explaining when and why edge weakening can enhance synchronization, especially for general directed weighted or signed networks, is still largely missing.
Existing results are often of two different types. Some are derived for highly structured graph families, where the special topology permits sharp conclusions on synchronization or convergence, while others emphasize structural or spectral responses to network modification without providing an explicit edgewise local perturbation law for a general directed weighted setting. The present paper aims to bridge these two viewpoints.

In this work, we develop a perturbation-based spectral framework to systematically analyze synchronization in directed weighted networks. 
Our starting point is the counterintuitive phenomenon that weakening selected edges may increase network synchronizability in directed weighted networks. Moreover, in some nonlinear systems and parameter regimes, such spectral improvement may even be accompanied by a transition from a nonsynchronized regime to a synchronized one.
As illustrated in Fig.~\ref{fig:intro_second}, our numerical examples show that, for the specific second-order nonlinear systems considered here, reducing the weights of appropriately chosen edges can move the dynamics from an incoherent regime to a synchronized one.

To explain this phenomenon, we focus on the generalized algebraic connectivity $\kappa$ defined in~\eqref{eq:kappa}, which is the central spectral quantity in our analysis. For directed weighted networks, $\kappa$ plays the role of an algebraic-connectivity-type quantity while accounting for coupling asymmetry through the left eigenvector associated with the zero eigenvalue of the Laplacian. This makes $\kappa$ more than a convenient spectral proxy: it is also a dynamically meaningful quantity. In particular, earlier studies~\cite{liu2008synchronization,lu2009synchronisation} on synchronization with asymmetrical coupling matrices highlighted the role of the left zero-eigenvector in directed settings, while works on directed second-order consensus~\cite{yu2009second,yu2010some} identified generalized algebraic-connectivity-type quantities as key parameters in consensus conditions. For these reasons, $\kappa$ provides a natural lens through which to study how local edge modifications in directed weighted networks influence global synchronization-related performance.

Our main theoretical contribution is an explicit first-order perturbation formula for the variation of $\kappa$ under edge-weight perturbations (see Theorem~\ref{thm:kappa}). 
More specifically, we show that the first-order sensitivity of $\kappa$ admits a natural and interpretable decomposition into two terms: a \emph{directed cut-energy term}, which captures the local effect of asymmetric flow imbalance across perturbed edges, and a \emph{stationary redistribution term}, which describes the global effect induced by the perturbation of the invariant measure. 
This decomposition reveals that, in directed networks, the synchronization role of an edge is determined not only by the local disagreement it carries, but also by how its perturbation redistributes stationary mass throughout the network. 

Building on this $\kappa$-sensitivity analysis, we further show that the first-order sensitivities are additive across multiple edge perturbations. 
This additivity makes it possible to rank and combine edge modifications in a principled way, and it leads naturally to an effective sensitivity-guided algorithmic framework for network optimization. 
Based on this idea, we develop a greedy method for edge modification, which provides a unified treatment of edge-weight reduction (see Algorithm~\ref{alg:iterative_reduce}), edge deletion and negative-edge insertion (see Algorithm~\ref{alg:discrete_modify}). 
The same framework can also be used to guide weight enhancement (see Section~\ref{sec:Discussion}), thereby allowing both weakening and strengthening strategies to be compared within a common spectral perspective.

Extensive experiments on synthetic and real directed networks demonstrate the effectiveness of the proposed approach. 
The results show that the $\kappa$-sensitivity framework can identify a set of critical edges whose weakening yields large improvements in generalized algebraic connectivity (see Fig.~\ref{fig:greedy}). 
Moreover, simulations of first- and second-order nonlinear consensus dynamics show that these spectral improvements are accompanied by markedly faster convergence and, in some cases, by a clear transition from incoherence to synchronization (see Fig.~\ref{fig:first_nonlinear} and Fig.~\ref{fig:second_nonlinear}). 
These observations indicate that the proposed sensitivity is not only useful for local spectral analysis, but also highly informative for nonlinear synchronization behavior in practice.

Collectively, our results provide a unified perturbation-based spectral perspective on synchronization optimization in directed weighted networks. 
The proposed framework naturally extends to a broader class of edge-level operations, including edge strengthening, edge deletion, and negative-edge insertion.
This offers a new viewpoint for the analysis, design, and control of directed networked dynamical systems.

The main contributions of this work are summarized as follows:
\begin{itemize}

\item We establish a perturbation-based spectral framework for directed weighted networks and derive an explicit first-order sensitivity formula for the generalized algebraic connectivity $\kappa$. 
This formula quantifies the variation of $\kappa$ under edge-level perturbations and yields an interpretable decomposition into a directed cut-energy term and a stationary redistribution term.

\item We show that the first-order $\kappa$-sensitivity is additive across multiple edge perturbations. 
Based on this property, we develop a backtracking-based greedy algorithm that uses the local perturbation formula as a ranking heuristic for edge-level modifications.

\item We validate the proposed framework on synthetic and real directed networks, and further test it on first- and second-order nonlinear consensus dynamics. 
In the reported experiments, sensitivity-guided edge weakening consistently improves $\kappa$ and convergence performance, and in some tested instances is accompanied by a transition from a nonsynchronized regime to a synchronized one.

\end{itemize}

The remainder of this paper is organized as follows. 
Section~\ref{sec:RelatedWork} reviews related work. 
Section~\ref{sec:Preliminaries} introduces the problem formulation and preliminaries. 
Section~\ref{sec:Perturbation} presents the perturbation analysis and the main theoretical results. 
Section~\ref{sec:Algorithm} develops the proposed greedy and heuristic algorithms for edge modification. 
Section~\ref{sec:Results} reports numerical results on real directed networks. 
Section~\ref{sec:Discussion} provides further discussion. 
Finally, Section~\ref{sec:Conclusion} concludes the paper.

\section{Related Work}
\label{sec:RelatedWork}

This paper is situated at the intersection of spectral graph theory, directed consensus, and synchronization-oriented network modification. 
To provide a clear research context for this interdisciplinary topic, we first review the classical spectral intuition developed for undirected networks, then explain why directed networks require different spectral quantities, and finally discuss network modification, sensitivity analysis, and edge-level design. 
Table~\ref{tab:related_work_comparison} summarizes the positioning of the proposed framework relative to representative related work.

\subsection{Algebraic Connectivity and Monotone Intuition}
\label{subsec:rw_undirected}

Algebraic connectivity is one of the central quantities in spectral graph theory. 
For an undirected weighted graph, the second smallest eigenvalue of the Laplacian, usually denoted by $\lambda_2$, characterizes graph connectivity in a quantitative way and is closely related to cuts, diffusion, random walks, robustness, and convergence of agreement processes~\cite{Newman2018Networks,fiedler1973algebraic,GodsilRoyle2001Algebraic}. 
In networked dynamical systems, this quantity also provides an important bridge between graph structure and collective behavior. 
For example, in consensus and cooperative control, larger algebraic connectivity is typically associated with faster convergence of agreement dynamics~\cite{Li2017Cooperative,olfati2007consensus}. 

A consequence of this classical theory is a strong monotone intuition: 
adding undirected edges or increasing symmetric edge weights should improve connectivity and is often expected to facilitate convergence or synchronizability. 
This intuition is mathematically justified for pure edge strengthening in undirected graphs, because strengthening a single undirected edge $\{i,j\}$ adds the positive semidefinite rank-one Laplacian contribution 
$\varepsilon (e_i-e_j)(e_i-e_j)^\top$. 
More generally, strengthening multiple undirected edges adds a sum of such positive semidefinite contributions, and hence cannot decrease the algebraic connectivity. 
This explains why many network-design methods for undirected graphs focus on edge addition, edge rewiring, or related graph-augmentation strategies for maximizing algebraic connectivity~\cite{ghosh2006growing,sydney2013optimizing}.

However, this monotone picture is fundamentally tied to symmetry. 
Once interactions become directed, an edge-weight perturbation no longer corresponds to a symmetric positive semidefinite update. 
Increasing one directed edge may change not only the local interaction across that edge, but also the global balance of influence induced by the directed flow. 
Therefore, the undirected principle that ``more coupling improves connectivity'' cannot be directly transferred to directed weighted networks. 
This observation motivates the need for directed spectral quantities and perturbation tools that explicitly account for asymmetry.

\subsection{Directed Consensus and Generalized Connectivity}
\label{subsec:rw_directed}

Directed networks arise naturally in many real systems, including social influence networks, biological regulation, transportation systems, communication networks, and multi-agent coordination. 
For such networks, the graph Laplacian is generally nonsymmetric, and its left and right eigenvectors play different roles. 
A key feature of directed consensus and synchronization is that the left eigenvector associated with the zero eigenvalue determines the invariant weighted average or stationary mass distribution of the network~\cite{olfati2007consensus}. 
Thus, directed synchronization is not governed only by pairwise coupling strengths; it also depends on how influence is globally distributed across the network.

Several works have developed synchronization and consensus theory for directed or asymmetrically coupled networks. 
Liu and Chen~\cite{liu2008synchronization} studied synchronization in nonlinearly coupled networks with asymmetric coupling matrices and used the left zero-eigenvector $\xi$ and the weighted matrix $\Xi=\operatorname{diag}(\xi)$ in Lyapunov analysis. 
Lu and Chen~\cite{lu2009synchronisation} further developed synchronization criteria for complex networks with directed topologies. 
For second-order multi-agent systems, Yu, Chen, Cao, and Kurths~\cite{yu2009second,yu2010some} introduced generalized algebraic-connectivity-type quantities involving the symmetrized matrix $(\Xi L+L^\top \Xi)/2$
and showed that such quantities enter sufficient and sometimes necessary conditions for directed second-order consensus. 
These works provide the main dynamical motivation for the generalized algebraic connectivity $\kappa$ used in this paper.

More broadly, directed spectral graph theory has shown that directed graphs require spectral objects beyond the standard undirected Laplacian. 
By incorporating the Perron vector of a directed random walk, Chung~\cite{chung2005laplacians} introduced a stationary-distribution-weighted Hermitian Laplacian and established a Cheeger inequality for directed graphs. 
This perspective is consistent with the role of the left zero-eigenvector in directed consensus and synchronization. 
The generalized algebraic connectivity $\kappa$ used in this paper follows the same principle: it combines a Laplacian-type quadratic form with the stationary weighting induced by the directed network. 
Thus, $\kappa$ is not merely a numerical proxy for connectivity, but a synchronization-relevant directed spectral quantity.

Nevertheless, most existing results on directed consensus and synchronization use generalized connectivity quantities as global performance or convergence indicators. 
They clarify why such quantities are useful for certifying consensus, synchronization, or convergence rates, but they do not provide an edgewise perturbation law describing how a local directed edge modification changes the generalized algebraic connectivity. 
This edge-level sensitivity question is the main focus of the present paper.

\subsection{Signed and Asymmetric Network Interactions}
\label{subsec:rw_signed}

Many real-world networks contain both cooperative and antagonistic interactions. 
Signed networks have therefore become an important framework for modeling disagreement, polarization, opinion separation, and antagonistic coordination~\cite{altafini2013consensus,Meng2022Disagreement}. 
In signed consensus theory, negative edges do not simply destroy collective behavior; rather, under structural-balance conditions, they may lead to bipartite consensus, where agents agree in modulus but split into two groups with opposite signs~\cite{altafini2013consensus}. 
More generally, signed interactions can generate structural-balance-dependent consensus or disagreement behavior. 
This line of work demonstrates that the sign and direction of an interaction can qualitatively change network dynamics.

For synchronization, negative interactions have also been shown to produce counterintuitive effects. 
Nishikawa and Motter~\cite{nishikawa2010network} showed that negative interactions and compensatory structures can improve synchronizability in certain networks, challenging the conventional intuition that synchronizing interactions should always be positive or purely cooperative. 
Similarly, directed link manipulation can be exploited to improve synchronization. 
Zeng, L{\"u}, and Zhou~\cite{zeng2012manipulating} showed that PageRank-guided link addition, removal, and rewiring can enhance directed-network synchronizability by exploiting hierarchy and core--receptor structures.

These results collectively suggest that synchronization in asymmetric networks is governed by a more delicate balance than in undirected graphs. 
Local changes that appear harmful from the viewpoint of coupling strength may still improve a global spectral or dynamical quantity by reshaping the effective influence pattern of the network. 
However, most works in this direction focus on global synchronizability landscapes, structural heuristics, or specific classes of link modifications. 
They do not provide a general first-order edgewise formula explaining how a directed edge perturbation changes a generalized algebraic-connectivity-type quantity.

\subsection{Synchronization-Oriented Network Modification}
\label{subsec:rw_modification}

A growing body of work has studied how structural modifications affect synchronization and consensus. 
In undirected networks, edge addition, edge rewiring, and related graph-augmentation strategies are natural tools for improving algebraic connectivity and convergence rates~\cite{ghosh2006growing,sydney2013optimizing}. 
In directed networks, however, the effect of edge modification can be non-monotone and topology-dependent.

Several studies have reported such nontrivial effects. 
Zhang, Chen, and Mo~\cite{zhang2017effect} analyzed the effect of adding edges to consensus networks with directed acyclic graphs and showed that arc addition can influence convergence in ways determined by the special graph structure. 
Gao et al.~\cite{gao2023effects,gao2023effects2} studied leader--follower multi-agent systems and showed that adding arcs or changing arc weights can affect the consensus convergence rate.
In particular, decreasing certain arc weights may accelerate convergence in structured leader--follower networks. 
Bakrani et al.~\cite{bakrani2024cycle} investigated motif-level responses to link modification and showed that local cycle--star structures can explain certain spectral responses.

These results are important because they demonstrate that link modification in directed networks can have effects that differ sharply from undirected intuition. 
However, many of them rely on special graph families, such as directed acyclic graphs, leader--follower networks, or motif-based structures. 
The corresponding conclusions are often precise and informative within their specific settings, but they do not yield a general perturbative law for arbitrary directed weighted networks.

Another line of work uses structural descriptors and ranking heuristics, such as edge betweenness, assortativity, and PageRank-type scores, to characterize important nodes, edges, or mixing patterns in networks~\cite{girvan2002community,newman2002assortative,gleich2015pagerank}. 
Such metrics are useful for describing network structure and guiding heuristic interventions, but they do not directly quantify the first-order variation of a synchronization-relevant spectral quantity. 
This limitation motivates the sensitivity-based approach developed in this paper.

\subsection{Perturbation and Spectral Network Design}
\label{subsec:rw_sensitivity}

Perturbation analysis provides a natural way to connect local structural changes with global spectral performance. 
For symmetric matrices and undirected Laplacians, eigenvalue perturbation formulas directly relate edge-weight changes to Fiedler-vector differences. 
This connection has been used in algebraic-connectivity maximization, graph augmentation, and network design~\cite{ghosh2006growing,sydney2013optimizing}. 
For directed networks, however, perturbation analysis is more subtle because the stationary distribution and possible eigenvalue branch switching must be considered.

Existing studies have used matrix perturbation theory to analyze synchronization enhancement in networks. 
For example, He et al.~\cite{he2019perturbation} compared several enhancement methods, including link removal, node removal, hub division, pull control, and pinning control, by deriving perturbation estimates for the ordinary Laplacian eigenvalues $\lambda_2$ and $\lambda_N$. 
Their results show that perturbation analysis can provide useful local predictions for synchronization-related spectral changes. 
However, their analysis does not address the generalized algebraic connectivity of directed weighted networks, where the invariant measure also changes under edge perturbations. 
In parallel, structural genericity results for Laplacian spectra indicate that simple eigenvalues and nondegenerate Fiedler structures are natural rather than exceptional in many weighted-graph settings~\cite{PPP18}. 
These observations motivate our local differentiable analysis of $\kappa$ under suitable nondegeneracy assumptions.

The present paper builds on this perturbation viewpoint but differs from existing spectral-design work in two respects. 
First, we focus on the generalized algebraic connectivity $\kappa$ of directed weighted networks, rather than on the classical undirected $\lambda_2$ or only on the  minimum real part of a nonsymmetric Laplacian eigenvalue $\gamma$. 
Second, our formula explicitly separates the first-order response into a local directed cut-energy term and a global stationary redistribution term. 
The latter term has no analogue in the classical undirected setting and is precisely the mechanism that allows edge weakening to improve $\kappa$ in directed networks.

\begin{table*}[!t]
\centering
\caption{Comparison with representative related work on synchronization, directed connectivity, and network modification.}
\label{tab:related_work_comparison}
\scriptsize
\setlength{\tabcolsep}{3pt}
\renewcommand{\arraystretch}{1.15}
\begin{tabularx}{\textwidth}{
>{\raggedright\arraybackslash}p{1.85cm}
>{\raggedright\arraybackslash}p{1.85cm}
>{\raggedright\arraybackslash}p{2.05cm}
>{\raggedright\arraybackslash}p{2.35cm}
>{\raggedright\arraybackslash}p{4cm}
>{\raggedright\arraybackslash}X}
\toprule
Reference  & Network type & Modification / design focus & Main spectral or dynamical quantity & Main insight & Difference from this work \\
\midrule

\cite{fiedler1973algebraic} 
& Undirected graphs 
& None; spectral characterization 
& Algebraic connectivity $\lambda_2$ 
& Introduced $\lambda_2$ as a fundamental measure of graph connectedness. 
& Does not address directed asymmetry or signed weights. \\
\midrule
\cite{olfati2007consensus}
& Multi-agent networks
& Consensus and cooperation
& Laplacian spectrum and left zero-eigenvector
& Shows that strongly connected digraphs reach weighted consensus determined by the left zero-eigenvector.
& Focuses on consensus values; our work studies how directed edge perturbations change $\kappa$ through local cut-energy and stationary redistribution. \\
\midrule
\cite{nishikawa2010network}
& Weighted oscillator networks
& Link removal and negative interactions
& Synchronizability landscape
& Revealed compensatory structures and positive effects of negative interactions.
& Provides global structural insight, but not a general first-order edgewise formula for directed $\kappa$. \\
\midrule
\cite{zeng2012manipulating}
& Directed unweighted networks
& Link addition, removal, and rewiring
& Real-part Laplacian eigenratio
& Uses PageRank-based centrality to guide link manipulation for improving directed-network synchronizability.
& Heuristic structural manipulation; our work provides an edgewise $\kappa$-sensitivity theory for directed weighted perturbations. \\
\midrule
\cite{zhang2017effect}
& Directed acyclic consensus networks
& Arc addition
& Consensus convergence rate
& Characterized how adding edges affects convergence in DAG-based consensus networks.
& Relies on special graph topology; our framework applies under spectral nondegeneracy assumptions. \\
\midrule
\cite{he2019perturbation}
& Oscillator networks
& Synchronization enhancement methods
& Perturbations of $\lambda_2$ and $\lambda_N$
& Compares several enhancement methods using Laplacian eigenvalue perturbation.
& Does not address general directed $\kappa$ or stationary redistribution under edge perturbations. \\
\midrule
\cite{gao2023effects,gao2023effects2}
& Leader--follower directed networks
& Arc addition and arc-weight change
& Consensus convergence rate
& Showed that changing or reducing selected arc weights can improve convergence in structured digraphs.
& Strongly tied to leader--follower topology; our approach is a general perturbative framework. \\
\midrule
\cite{bakrani2024cycle}
& Motif-based networks
& Link modification
& Motif spectral response
& Explained network response through cycle--star motifs.
& Motif-level mechanism; our formula gives a whole-network edgewise sensitivity for $\kappa$. \\
\midrule
\cite{liu2008synchronization,lu2009synchronisation}
& Directed coupled networks
& Synchronization criteria
& Left zero-eigenvector and weighted Lyapunov analysis
& Highlighted the role of the invariant vector in directed synchronization.
& Motivates our use of $\xi$, but does not quantify how edge perturbations change $\kappa$. \\
\midrule
\cite{yu2009second,yu2010some}
& Directed second-order multi-agent systems
& Consensus conditions
& Generalized algebraic connectivity
& Showed that generalized connectivity enters directed second-order consensus conditions.
& Uses $\kappa$-type quantities as global criteria; our work derives their edgewise sensitivity. \\
\midrule
\cite{ghosh2006growing,sydney2013optimizing}
& Undirected graphs
& Edge addition and rewiring
& Algebraic connectivity $\lambda_2$ 
& Uses Fiedler-vector-based rules to improve $\lambda_2$ by adding or rewiring edges.
& Relies on the symmetry of undirected Laplacians; our work studies directed edge perturbations of $\kappa$ and identifies the stationary redistribution effect. \\
\midrule
\cite{altafini2013consensus}
& Signed networks
& Antagonistic interactions
& Signed consensus and structural balance
& Demonstrated that negative interactions can induce nonstandard collective behavior.
& Focuses on structural-balance conditions; our work treats negative-edge insertion as a $\kappa$-sensitivity-guided edge perturbation. \\
\midrule
\cite{girvan2002community,newman2002assortative,gleich2015pagerank}
& General networks
& Structural descriptors and ranking heuristics
& Betweenness, assortativity, PageRank
& Characterize node/edge importance and network mixing patterns.
& They do not directly predict the first-order change of $\kappa$ under directed edge perturbations. \\
\midrule
This paper
& Directed weighted and signed networks
& Edge weakening, deletion, negative insertion, and strengthening
& Generalized algebraic connectivity $\kappa$
& Derives an explicit edgewise first-order sensitivity formula and decomposes it into directed cut-energy and stationary redistribution.
& Provides a unified perturbation mechanism and sensitivity-guided network design algorithms. \\
\bottomrule
\end{tabularx}
\end{table*}

\subsection{Lessons Learned and Positioning of This Work}
\label{subsec:rw_lessons}

The above literature suggests several lessons that motivate the present work.

\begin{itemize}[leftmargin=*]
    \item In undirected networks, algebraic connectivity gives a clean monotone principle: symmetric edge strengthening improves the connectivity-related spectral quantity.
    
    \item This monotone principle is not generally valid in directed networks because a directed edge perturbation modifies both local coupling and the global stationary distribution.
    
    \item Directed consensus and synchronization theory shows that the left zero-eigenvector and the generalized algebraic connectivity $\kappa$ are dynamically meaningful quantities in asymmetric networks.
    
    \item Existing counterintuitive examples, including link removal, weight reduction, directionality manipulation, and negative interactions, indicate that weaker local coupling may improve global synchronization-related performance.
    
    \item Many available explanations are topology-specific, relying on special graph families, leader--follower structures, motifs, or numerical synchronizability landscapes.
    
    \item A general edgewise perturbation law is therefore needed to explain when and why a local directed edge modification improves a global synchronization-related spectral quantity.
\end{itemize}

Motivated by these lessons, this paper develops a perturbation-based $\kappa$-sensitivity framework for directed weighted networks. 
Our main theoretical contribution is an explicit first-order formula for the variation of $\kappa$ under edge-weight perturbations. 
The formula decomposes the sensitivity into a directed cut-energy term and a stationary redistribution term. 
This decomposition explains why weakening a carefully selected directed edge may increase $\kappa$, and it also provides a principled score for edge weakening, edge deletion, negative-edge insertion, and budget-constrained edge strengthening. 
Unlike topology-specific results, the proposed framework applies to general directed weighted networks under spectral nondegeneracy assumptions. 
Unlike purely structural heuristics, it directly quantifies the local spectral effect of edge-level interventions.

\section{Problem Formulation and Preliminaries}\label{sec:Preliminaries}

\subsection{Graph Model and Notation}

We consider a directed weighted graph $G=(V,E,A)$ with $n$ vertices, where 
$V=\{1,\ldots,n\}$
is the vertex set, $E\subseteq V\times V$ is the directed edge set, and 
$A=(a_{ij})_{i,j=1}^n$
is the weighted adjacency matrix. Throughout this paper, we restrict attention to \emph{simple directed graphs}, namely graphs without self-loops and with at most one directed edge between any ordered pair of vertices.

For a directed edge $j\to i\in E$, the entry $a_{ij}$ denotes its weight. In general, $A$ is not symmetric, and the weights $a_{ij}$ are allowed to be positive, zero, or negative. For each vertex $i$, we define its signed in-degree and signed out-degree by
\[
k_i^{\mathrm{in}}=\sum_{j=1}^n a_{ij},
\qquad
k_i^{\mathrm{out}}=\sum_{j=1}^n a_{ji}.
\]
The Laplacian matrix associated with $G$ is defined as
\begin{align}\label{eq:L}
L = D-A,
\qquad
D=\mathrm{diag}(k_1^{\mathrm{in}},\ldots,k_n^{\mathrm{in}}).
\end{align}
By construction, the rows of $L$ sum to zero, and hence
$L\mathbf{1}=0$,
where $\mathbf{1}=(1,\ldots,1)^\top\in\mathbb{R}^n$.

Throughout the paper, we impose the following standing assumptions.

\begin{assumption}\label{ass:A1}
The Laplacian matrix $L$ has a simple eigenvalue at $0$, and every other eigenvalue $\lambda\in\sigma(L)\setminus\{0\}$ satisfies
\[
\Re(\lambda)>0.
\]
\end{assumption}

\begin{assumption}\label{ass:A2}
The left eigenvector $\xi$ associated with the eigenvalue $0$ of $L$ can be chosen entrywise positive and is normalized by $\mathbf{1}^\top \xi =1$.
\end{assumption}

\begin{remark}
For directed graphs with nonnegative weights, Assumptions~\ref{ass:A1} and \ref{ass:A2} are satisfied under standard connectivity conditions. 
For example, if the graph contains a directed spanning tree, then $0$ is a simple eigenvalue of $L$ and all other eigenvalues have positive real parts; if the graph is strongly connected, then the corresponding left eigenvector associated with $0$ can be chosen strictly positive. 
Since the present paper allows signed edge weights, these graph-theoretic conditions are no longer sufficient in general, and we therefore impose Assumptions~\ref{ass:A1} and \ref{ass:A2} directly.
\end{remark}

For directed weighted networks, the classical algebraic connectivity of undirected graphs is no longer directly applicable. 
Following the generalized spectral framework used in directed consensus and synchronization analysis~\cite{liu2008synchronization,lu2009synchronisation,yu2009second,yu2010some}, we consider the following symmetrized matrix associated with $L$ and $\xi$.

Let $\Xi=\mathrm{diag}(\xi_1,\ldots,\xi_n)$,
and define
\begin{align}\label{eq:M}
\widehat{\Xi L}
=
\frac{\Xi L+L^\top \Xi}{2},
\qquad
M
=
\Xi^{-1/2}\widehat{\Xi L}\,\Xi^{-1/2}.
\end{align}
The matrix $M$ is real symmetric. To ensure that the generalized algebraic connectivity is well-defined and nondegenerate, we impose the following assumption.

\begin{assumption}\label{ass:A3}
The matrix $M$ is positive semidefinite, and its second smallest eigenvalue
\[
\kappa\coloneqq \lambda_2(M)
\]
satisfies $\kappa>0$ and is simple.
\end{assumption}

\begin{remark}
Assumptions~\ref{ass:A1}--\ref{ass:A3} may be viewed as nondegeneracy assumptions on the directed weighted network. In the nonnegative-weight setting, part of this structure is supported by known genericity results on graph Laplacians. In particular, Poignard, Pereira, and Pade~\cite{PPP18} proved that, for connected undirected weighted graphs, having simple Laplacian eigenvalues is structurally generic, and so is the property that the Fiedler vector has no zero entries. They also showed that, for strongly connected directed weighted graphs, being diagonalizable with simple Laplacian eigenvalues is structurally generic. These results do not directly imply Assumptions~\ref{ass:A1}--\ref{ass:A3} in our signed directed setting, since our analysis involves the symmetrized matrix $M$,
and we also allow signed edge weights. Nevertheless, they indicate that the simplicity and nondegeneracy requirements imposed here are natural from a structural point of view, rather than exceptional spectral restrictions. In particular, the simplicity of the synchronization-relevant spectral quantity is generic in the positive-weight case, which supports the interpretation of Assumption~\ref{ass:A3} as a nonpathological spectral hypothesis.
\end{remark}

Under Assumption~\ref{ass:A3}, the eigenvalues of $M$ can be ordered as 
$0=\lambda_1(M)<\lambda_2(M) < \lambda_{3}(M) \le \cdots \le \lambda_n(M)$,
and the quantity
\begin{align}\label{eq:kappa}
\kappa\coloneqq \lambda_2(M)
\end{align}
is called the \emph{generalized algebraic connectivity} of the directed weighted network.

Let $v$ be the normalized eigenvector of $M$ associated with $\kappa$, i.e.,
$Mv=\kappa v$, $\|v\|_2=1$,
and define $y=\Xi^{-1/2}v$.

\subsection{Problem Statement}

We are interested in how edge-level modifications of a directed weighted network affect its synchronization-related spectral properties. 
In particular, we study the first-order sensitivity of the generalized algebraic connectivity $\kappa$ with respect to signed perturbations of selected directed edges.

Given a directed weighted graph $G=(V,E,A)$, our goals are:
\begin{itemize}
    \item to characterize the first-order variation of $\kappa$ under edge-level perturbations;
    \item to identify the edge-wise contributions that increase or decrease $\kappa$;
    \item to develop sensitivity-guided strategies for synchronization enhancement through edge modifications.
\end{itemize}

A central question motivating this work is whether counterintuitive modifications, such as weakening selected edges or introducing negative interactions, can improve synchronization-related spectral performance in directed weighted networks, and whether such improvements may be accompanied by better synchronization behavior in specific nonlinear systems.
To address this question, we consider perturbations supported on a directed edge set $F$ defined on the vertex set $V$, where $F$ is not necessarily a subgraph of the original graph $G$ and may include newly introduced edges. 
The perturbed adjacency matrix is defined by
\begin{align}\label{eq:Ae}
A(\epsilon)
=
A+\epsilon\sum_{j\to i\in F}E_{ij},
\end{align}
where $\epsilon\in\mathbb{R}$ is sufficiently small. 
Under this convention, positive $\epsilon$ corresponds to strengthening the selected edges or inserting positive weights on previously absent edges, whereas negative $\epsilon$ corresponds to weakening existing edges or inserting negative weights on previously absent edges.

Let $L(\epsilon)$, $\xi(\epsilon)$, $\Xi(\epsilon)$, $M(\epsilon)$, and $\kappa(\epsilon)$ denote the corresponding perturbed quantities induced by~\eqref{eq:Ae}. 
Our objective is to derive an explicit first-order perturbation formula for
\[
\partial_F\kappa
\coloneqq 
\left.\frac{d\kappa(\epsilon)}{d\epsilon}\right|_{\epsilon=0},
\]
to interpret its structural meaning, and to use it for synchronization-oriented network modification.

\section{Perturbation Analysis and Theoretical Results}\label{sec:Perturbation}

This section develops a perturbation-based analysis of the generalized algebraic connectivity $\kappa$ under the edge-weight perturbation model~\eqref{eq:Ae}. 
Our main goal is to derive an explicit first-order sensitivity formula and to show that it admits a decomposition into two interpretable terms. 
This decomposition will clarify how local edge imbalance and global stationary redistribution jointly determine the synchronization role of each directed interaction.

\subsection{First-Order Perturbation Formula}

We consider the perturbation model introduced in~\eqref{eq:Ae}.
For notational convenience, we write
$\partial_e \kappa$ instead of $\partial_{\{e\}} \kappa$ to denote the contribution of a single directed edge perturbation $e=(j\to i)$.

The following theorem gives the first-order perturbation formula for $\kappa$.

\begin{theorem}\label{thm:kappa}
    Under Assumptions~\ref{ass:A1}--\ref{ass:A3}, for every directed edge set $F$ on $V$, the generalized algebraic connectivity $\kappa$ satisfies the first-order perturbation formula
    \begin{align}\label{eq:de1}
        \partial_F \kappa
        =
        \underbrace{
        \sum_{j\to i\in F}
        \xi_i\,y_i\,(y_i-y_j)
        }_{\text{Directed Cut Energy}}
        +
        \underbrace{
        y^\top
        (\partial_F \Xi)
        (L-\kappa I)y
        }_{\text{Stationary Redistribution Effect}},
    \end{align}
    where $y \coloneqq \Xi^{-1/2}v$ and $v$ is the normalized eigenvector of $M$ associated with $\kappa$.
\end{theorem}

The proof is deferred to the Appendix.

\begin{remark}
Under the sign convention in~\eqref{eq:Ae}, the derivative $\partial_F\kappa$ describes the first-order effect of \emph{increasing} the weights of the directed edges in $F$. 
Consequently, if one is interested in \emph{weakening} a set of edges, corresponding to $\epsilon<0$, then the first-order change in $\kappa$ has the opposite sign. 
In particular, an edge set with $\partial_F\kappa<0$ is beneficial to weaken, since a negative perturbation then increases $\kappa$ to first order.
\end{remark}

\subsection{Structural Interpretation of the Sensitivity}

The perturbation formula~\eqref{eq:de1} shows that the first-order sensitivity of $\kappa$ consists of two qualitatively different contributions. 
The first is a local edge term depending only on the perturbed edges and the generalized Fiedler geometry, while the second is a global correction arising from the perturbation of the stationary distribution.

\subsubsection{Directed Cut Energy}

The first term in~\eqref{eq:de1},
\[
\sum_{j\to i\in F}\xi_i\,y_i\,(y_i-y_j),
\]
can be viewed as a directed cut-energy term. 
It depends only on edgewise differences of the transformed eigenvector $y$ across the perturbed directed edges and therefore reflects the local effect of the perturbation on the asymmetric flow structure of the network.

More specifically, the factor $(y_i-y_j)$ measures the disagreement across the directed edge $j\to i$, while the weight $\xi_i y_i$ incorporates both the stationary distribution and the spectral structure of the generalized Fiedler mode. 
Hence, this term quantifies how strongly the perturbed edges contribute to the local synchronization geometry. 
Under the perturbation convention~\eqref{eq:Ae}, a positive value of this term indicates that strengthening the corresponding edge set tends to increase $\kappa$, whereas a negative value indicates that weakening those edges tends to increase $\kappa$, and therefore improves the synchronization-related spectral quantity considered in this work.

\subsubsection{Stationary Redistribution Effect}

The second term in~\eqref{eq:de1},
\[
y^\top(\partial_F\Xi)(L-\kappa I)y,
\]
captures the effect of the perturbation on the stationary distribution $\xi$. 
Unlike the directed cut-energy term, which is purely local, this contribution reflects a global redistribution effect induced by the directed nature of the network.

In a directed graph, the stationary distribution depends on the entire network structure. 
As a consequence, even a localized edge-weight perturbation may alter $\xi$ globally, and this redistribution in turn feeds back into the value of $\kappa$. 
This phenomenon has no counterpart in the classical undirected setting, where $\Xi$ reduces to a scalar multiple of the identity and therefore remains invariant under edge perturbations.

\subsection{Vanishing of the Stationary Redistribution Effect}

The decomposition in Theorem~\ref{thm:kappa} naturally raises the question of when the second term vanishes, so that the first-order variation of $\kappa$ is determined entirely by the local directed cut-energy term.

\begin{proposition}\label{prop:CD}
Define
\begin{align}
C\coloneqq \Xi L-L^\top\Xi,
\qquad
D\coloneqq  (\partial_F\Xi)\Xi^{-1}.
\end{align}
Then the stationary redistribution term $R\coloneqq y^\top(\partial_F\Xi)(L-\kappa I)y$
can be rewritten as
\begin{align}\label{eq:RDC}
R
=
\frac12\, y^\top D C y.
\end{align}
\end{proposition}

The proof is given in the Appendix.

Proposition~\ref{prop:CD} immediately yields the following corollary.

\begin{corollary}\label{cor:vanish}
The stationary redistribution term vanishes if either of the following conditions holds:
\begin{itemize}
    \item $C=0$, that is, $\Xi L=L^\top\Xi$,
    which corresponds to the detailed-balance (reversible) case;

    \item $D=0$, equivalently, $\partial_F\xi=0$,
    i.e., the stationary distribution is invariant to first order under the perturbation supported on $F$.
\end{itemize}
\end{corollary}

\begin{remark}
In the nonnegative-weight setting, the condition $C=\Xi L-L^\top\Xi=0$
has a natural combinatorial interpretation. Indeed, for $i\neq j$ one has $C_{ij}=-(\xi_i a_{ij}-\xi_j a_{ji})$,
so $C=0$ is equivalent to the detailed-balance relations $\xi_i a_{ij}=\xi_j a_{ji}$, $\forall\, i,j$.
Thus, the stationary flux through each pair of opposite directed edges is balanced.

A classical consequence is Kolmogorov's cycle condition: for every directed cycle
$i_1\to i_2\to \cdots \to i_m \to i_1$, with $i_{m+1}=i_1$, one must have
\[
\prod_{r=1}^{m} a_{i_{r+1}i_r}
=
\prod_{r=1}^{m} a_{i_ri_{r+1}}.
\]
Indeed, applying the detailed-balance relation to the pair
$(i_{r+1},i_r)$ gives
\[
\xi_{i_{r+1}}a_{i_{r+1}i_r}
=
\xi_{i_r}a_{i_ri_{r+1}},
\qquad r=1,\dots,m.
\]
Multiplying these identities around the cycle, all $\xi_{i_r}$ cancel
telescopically, yielding the above product equality.

Conversely, if the graph is strongly connected, all weights are nonnegative,
and the cycle condition holds for every directed cycle, then one can recover
a positive vector $\xi$ satisfying the detailed-balance equations. A standard
construction is as follows: fix a root node $r$, set $\xi_r=1$, and for any
node $i$ choose a directed path $r=i_0\to i_1\to \cdots \to i_k=i$.
Define
\[
\xi_i
:=
\prod_{s=0}^{k-1}
\frac{a_{i_si_{s+1}}}{a_{i_{s+1}i_s}}.
\]
The cycle condition guarantees that this definition is path-independent, and
one then verifies that $\xi_i a_{ij}=\xi_j a_{ji}$ for every pair $i,j$.
Hence, after normalization, $C=0$.

Therefore, in the strongly connected nonnegative-weight case, the matrix
condition $C=0$ is equivalent to the statement that every directed cycle has
the same multiplicative weight in the forward and reverse directions; see,
e.g., the classical reversibility criterion in \cite{kelly1979reversibility}.
\end{remark}

\begin{remark}
Since $C$ is skew-symmetric, one has $y^\top C y=0$, and hence $R=\frac12\,y^\top\!\left(D-\frac{1}{n}\operatorname{tr}(D)I\right)\!Cy$.
Therefore,
\begin{align*}
|R|
\le
\frac12\left\|D-\frac{1}{n}\operatorname{tr}(D)I\right\|
\,\|C\|\,\|y\|^2.
\end{align*}
This estimate shows that the stationary redistribution effect is small whenever the graph is close to reversible or the stationary redistribution is nearly uniform.
\end{remark}

\begin{figure}[htbp]
    \centering
    \includegraphics[width=\linewidth]{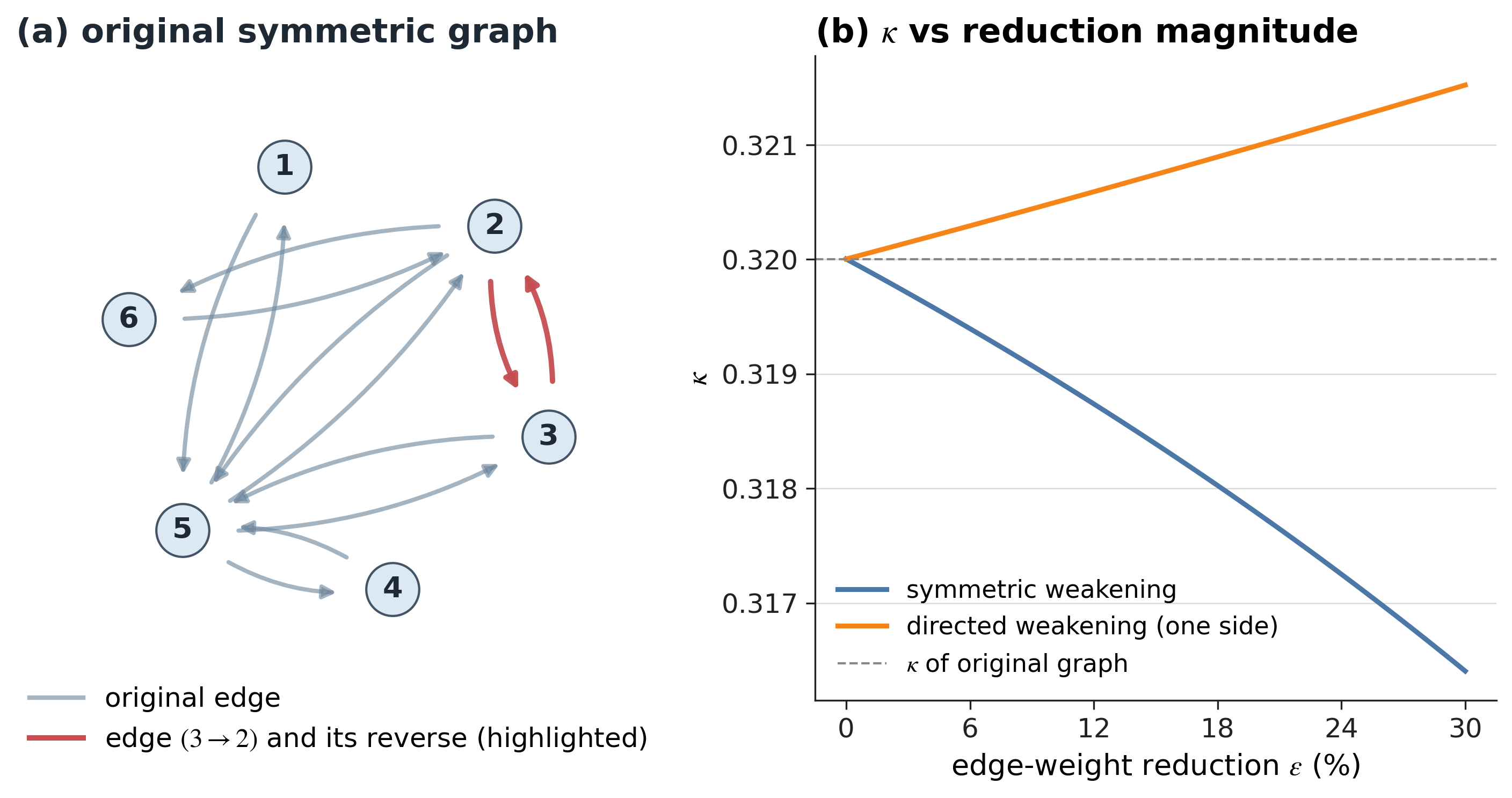}
    \caption{
Symmetric versus asymmetric edge-weight reduction on an originally undirected graph. 
(a) A symmetric weighted graph with the edge pair $(3 \leftrightarrow 2)$ highlighted. 
(b) Evolution of the generalized algebraic connectivity $\kappa$ versus the reduction magnitude $\epsilon$. 
Symmetric reduction (blue) decreases $\kappa$ monotonically, whereas one-sided reduction (orange) increases $\kappa$. 
The dashed line indicates the value of $\kappa$ for the original graph.
}
    \label{fig:remark7}
\end{figure}

\begin{remark}
The above decomposition recovers the classical monotonicity law for undirected graphs under \emph{symmetric} edge-weight perturbations. 
Indeed, in the undirected case one has $\Xi=\frac{1}{n}I$, so that $C=0$ and the stationary redistribution term vanishes. 
If a symmetric perturbation is applied to an undirected edge set $F$, then~\eqref{eq:de1} reduces to
\begin{align*}
\partial_F\kappa
=
\sum_{\{i,j\}\in F}\frac{1}{n}(y_i-y_j)^2
\ge 0,
\end{align*}
implying that strengthening (weakening) edges can only increase (decrease) the algebraic connectivity to first order.

Fig.~\ref{fig:remark7} shows that this monotonicity relies on symmetry. 
Once symmetry is broken, even by perturbing only one direction of an edge, the graph becomes genuinely directed, and the monotonicity may no longer hold. 
In particular, weakening an edge can then increase $\kappa$, indicating that such improvements are inherently tied to directional asymmetry.
\end{remark}

\subsection{Discussion of the Algebraic Connectivity $\gamma$}

Besides $\kappa$, another spectral quantity commonly used in directed consensus is
\[
\gamma \coloneqq \min_{\lambda\in\sigma(L)\setminus\{0\}} \Re(\lambda),
\]
which characterizes the slowest exponential decay rate of the linear consensus dynamics. 
The perturbation framework developed in this paper on $\kappa$ can also be carried out for $\gamma$, but this requires a stronger spectral assumption: the eigenvalue branch attaining the minimum real part must be locally unique, so that no branch switching occurs in a neighborhood of the reference network.

More precisely, assume that there exist $\varepsilon_0>0$ and a simple eigenvalue branch $\lambda_*(\epsilon)$ of $L(\epsilon)$ such that, for all $|\epsilon|<\varepsilon_0$, 
$\gamma(\epsilon)=\Re\bigl(\lambda_*(\epsilon)\bigr)$.
Let $x$ and $p$ denote the corresponding right and left eigenvectors, normalized by
$Lx=\lambda_*x$, $p^*L=\lambda_*p^*$, $p^*x=1$.
Then the first-order variation of $\gamma$ under the perturbation supported on $F$ is given by
\begin{align}\label{eq:degamma}
\partial_F\gamma
=
\sum_{j\to i\in F}
\Re\!\bigl(\overline{p_i}(x_i-x_j)\bigr).
\end{align}
Under the above local spectral-separation hypothesis, $\partial_F\gamma$ is the derivative of the actual quantity $\gamma$, rather than merely the derivative of the real part of a selected simple eigenvalue branch.
Accordingly, the single-edge $\gamma$-sensitivity is $\partial_e\gamma
\coloneqq
\Re\!\bigl(\overline{p_i}(x_i-x_j)\bigr)$, $e=(j\to i)$,
and the total sensitivity is additive: $\partial_F\gamma=\sum_{e\in F}\partial_e\gamma$.
In contrast to the $\kappa$-sensitivity, the $\gamma$-sensitivity depends only on the left and right eigenvectors associated with the relevant eigenvalue branch of $L$ and does not involve a stationary redistribution term.

\section{Algorithms for Edge Modification}\label{sec:Algorithm}

In this section, we translate the perturbation analysis into practical algorithms for network modification. 
The design is guided by the $\kappa$-sensitivity, namely the first-order variation $\partial_F\kappa$, which provides a principled way to evaluate the effect of edge-level perturbations on synchronization-related spectral performance.

\subsection{Computability of $\kappa$-Sensitivity}

We first show that the sensitivity $\partial_F\kappa$ admits an explicit and efficiently computable form. 
We begin with the single-edge case and then extend the result to general edge sets.

\subsubsection{Single-edge Perturbation}

Let $e=(j\to i)$ be a directed edge. 
Under the perturbation convention~\eqref{eq:Ae}, a positive perturbation parameter $\epsilon$ increases the weight of $e$, while a negative perturbation weakens it. 
By Theorem~\ref{thm:kappa}, the corresponding first-order sensitivity is
\begin{align}\label{eq:de_alg}
\partial_e \kappa
=
\xi_i\,y_i\,(y_i-y_j)
+
y^\top (\partial_e \Xi)(L-\kappa I)y.
\end{align}

The first term, referred to as the \emph{directed cut energy}, depends only on local edge information and can be evaluated directly once $(\xi,y)$ are known.

The second term, called the \emph{stationary redistribution effect}, requires the first-order variation of the stationary distribution $\xi$. 
Differentiating the stationary equation $\xi^\top L=0$ under the perturbation of the single edge $e=(j\to i)$ yields
\begin{align}\label{eq:ls}
L^\top \partial_e \xi = -\xi_i (e_i-e_j),
\qquad
\mathbf{1}^\top \partial_e \xi = 0,
\end{align}
where $e_i$ denotes the $i$th canonical basis vector. 
The second constraint enforces the normalization $\mathbf{1}^\top \xi=1$ at first order. 
Once $\partial_e\xi$ is obtained, one sets $\partial_e\Xi=\mathrm{diag}(\partial_e\xi)$,
and the redistribution term is computed as $y^\top (\partial_e \Xi)(L-\kappa I)y$.

Therefore, the full edge sensitivity $\partial_e\kappa$ can be evaluated explicitly for any candidate edge.

\subsubsection{Multiple-edge Perturbation}

For a perturbation supported on a directed edge set $F$, the total sensitivity is additive:
\begin{align}\label{eq:additivity_kappa}
\partial_F \kappa = \sum_{e\in F}\partial_e \kappa.
\end{align}
This follows from the linearity of the perturbation model~\eqref{eq:Ae}. 
Indeed, both the Laplacian variation and the induced stationary variation decompose edgewise (which can be read from the proof of Theorem~\ref{thm:kappa} in the Appendix):
\begin{align*}
L' = \sum_{e\in F} L'_e,
\qquad
\partial_F \xi = \sum_{e\in F} \partial_e \xi,
\qquad
\partial_F \Xi = \sum_{e\in F} \partial_e \Xi.
\end{align*}
Consequently, the total first-order sensitivity is the sum of individual edge contributions, which enables efficient evaluation, ranking, and selection of candidate edge modifications.

\subsection{Numerical Verification of the Sensitivity Formulas}
\begin{figure}[htpb]
    \centering
    \includegraphics[width=\linewidth]{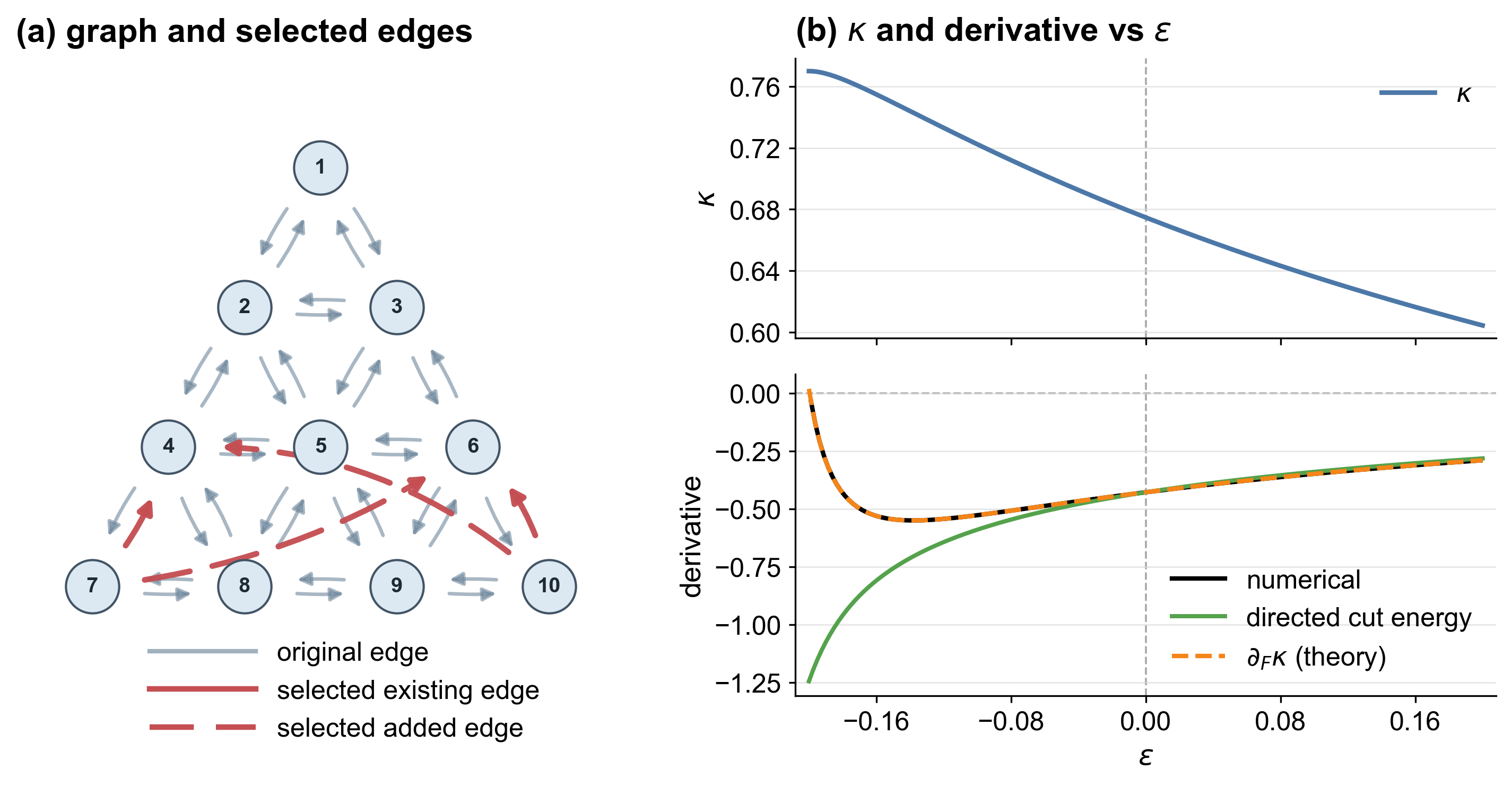}
    \caption{
Numerical verification of the sensitivity formula for $\kappa$. 
(a) Directed graph and selected edges. Blue edges denote the original graph, with unit weights except for inter-layer backward edges, which have weight $0.3$. 
Red solid edges are selected existing edges, and red dashed edges are selected added edges. 
(b) Evolution of $\kappa$ (top) and its derivative with respect to $\epsilon$ (bottom) under perturbations on the selected edges. 
The numerical derivative (black solid line) agrees closely with the theoretical prediction $\partial_F \kappa$ (orange dashed line). 
The directed cut-energy term (green curve) deviates from the full derivative, and the gap corresponds to the stationary redistribution effect.
}
    \label{fig:derivative}
\end{figure}
We validate the correctness of the sensitivity formulas on a directed weighted network with $10$ nodes, as shown in Fig.~\ref{fig:derivative}(a). 
The blue edges represent the original graph. 
All edges have unit weight except for the inter-layer backward edges (from a lower layer to an upper layer), whose weights are set to $0.3$. 
Specifically, the backward edges are
\begin{align*}
& 2\to1,\;3\to1,\\
& 4\to2,\;5\to2,\;5\to3,\;6\to3,\\
& 7\to4,\;8\to4,\;8\to5,\;9\to5,\;9\to6,\;10\to6.
\end{align*}

We select a set of edges $F$ highlighted in red in Fig.~\ref{fig:derivative}(a), where solid red edges correspond to existing edges and dashed red edges correspond to newly added edges. 
Following the perturbation convention~\eqref{eq:Ae}, we apply a uniform perturbation $\epsilon$ to all edges in $F$.
Fig.~\ref{fig:derivative}(b) shows that, within a neighborhood of $\epsilon=0$, increasing the weights of the selected edges (i.e., $\epsilon>0$) leads to a decrease in $\kappa$, whereas decreasing their weights (i.e., $\epsilon<0$) results in an increase. 
This phenomenon highlights the nontrivial role of edge weights in directed networks and provides a concrete example where weakening certain edges enhances connectivity.
Fig.~\ref{fig:derivative}(b) (bottom) compares the numerical derivatives (black solid line) with the theoretical prediction (orange dashed line) for $\partial_F \kappa$. 
The two curves coincide almost perfectly, confirming the correctness of the derived sensitivity formula. 
The contribution of the directed cut-energy term (green curve) is also shown. 
A clear discrepancy between this term and the full derivative $\partial_F \kappa$ is observed, which corresponds to the stationary redistribution effect. 
This demonstrates that the sensitivity of $\kappa$ is influenced by global variations in the stationary distribution $\xi$ and cannot be captured solely by local edge information.

\subsection{Sensitivity-Guided Greedy Strategy}

\begin{algorithm}[h]
\caption{Iterative sensitivity-guided edge weakening with composite step size}
\label{alg:iterative_reduce}
\begin{algorithmic}[1]
\Require Directed weighted graph $A$; step size $s>0$; 
         derivative fraction $d_f\geq 0$; batch size $k\geq 1$;
         selection mode \texttt{choose\_mode} $\in\{$\texttt{topk, randomk, all}$\}$;
         weight threshold $\tau_w>0$; numerical tolerance $\varepsilon>0$;
         maximum number of iterations $T_{\max}$
\Ensure Modified graph $B$ with increased spectral metric $\kappa$

\State $B \gets A$
\For{$t = 1, \dots, T_{\max}$}
    \State Compute edge sensitivities $\partial_e\kappa$ for all admissible edges of $B$
    \State Define the candidate weakening set
    \begin{align*}
        \mathcal{P}(B) \coloneqq \{e=(j\to i) : \partial_e\kappa < 0,\; w_{ij} > \tau_w\}
    \end{align*}
    \If{$\mathcal{P}(B) = \varnothing$}
        \State \textbf{break}
    \EndIf
    \State Select  $\mathcal{S} \subseteq \mathcal{P}(B)$ 
           according to \texttt{choose\_mode}
           \Comment{topk: largest $|\partial_e\kappa|$; randomk: uniformly at random; all: $\mathcal{S}= \mathcal{P}(B) $}
    \State $\partial_{\max} \gets \max_{e\in\mathcal{S}}|\partial_e\kappa|$
    \State $\sigma \gets s$, \quad accepted $\gets$ false
    \While{$\sigma > \varepsilon$ \textbf{ and } accepted $=$ false}
        \State Construct $\widetilde{B} \gets B$
        \For{each $e=(j\to i)\in\mathcal{S}$ with $w_{ij}>0$}
            \State Compute composite step
            \begin{align*}
                \delta_{ij} \gets \min\!\left(\sigma\!\left(
                w_{ij}
                + d_f\frac{|\partial_e\kappa|}{\partial_{\max}}\right),\;w_{ij}\right)
            \end{align*}
            \State $\widetilde{B}_{ij} \gets \max(0,\, w_{ij} - \delta_{ij})$
        \EndFor
        \If{$\kappa(\widetilde{B}) > \kappa(B)$ \textbf{ and } $\widetilde{B}$ satisfies
            Assumptions~\ref{ass:A1}--\ref{ass:A3}}
            \State $B \gets \widetilde{B}$
            \State accepted $\gets$ true
        \Else
            \State $\sigma \gets \tfrac{1}{2}\,\sigma$
            \Comment{fixed bisection backtracking}
        \EndIf
    \EndWhile
    \If{accepted $=$ false}
        \State \textbf{break}
    \EndIf
\EndFor
\State \Return $B$
\end{algorithmic}
\end{algorithm}

The perturbation formulas above naturally lead to a sensitivity-guided greedy strategy for network modification. 
Since the first-order derivative quantifies the local effect of a small edge-weight perturbation, it can be used to rank candidate edges according to their potential contribution to increasing $\kappa$ and improving synchronization-related spectral performance.

If the goal is to \emph{weaken} edges in order to improve $\kappa$, then one should preferentially select edges with negative sensitivity scores, because a negative perturbation $\epsilon<0$ applied to such edges yields a positive first-order change in $\kappa$. 
Similarly, if the goal is to \emph{strengthen} edges, then one should select edges with positive sensitivity scores. 
This provides a unified spectral criterion for deciding whether an edge should be weakened, strengthened, removed, or inserted.

Algorithm~\ref{alg:iterative_reduce} presents a sensitivity-guided greedy algorithm tailored to the edge-weakening task. 
At each iteration, the algorithm computes all single-edge sensitivities and selects candidate edges whose weakening is predicted to increase $\kappa$.

For edge weakening, the admissible set consists of edges with negative sensitivity, since decreasing the weight of such edges leads to a positive first-order variation of $\kappa$. 
Specifically, for $e=(j\to i)$, if $\partial_e\kappa<0$, then a negative perturbation on $e$ increases $\kappa$. 
This motivates the candidate set
\[
\mathcal{P}(B)
\coloneqq
\{\,e=(j\to i):\partial_e\kappa<0,\; w_{ij}>\tau_w\,\},
\]
where $\tau_w>0$ avoids numerically negligible updates.

From $\mathcal{P}(B)$, a subset $\mathcal{S}$ is selected according to a prescribed mode (\texttt{all}, \texttt{topk}, or \texttt{randomk}). 
A trial graph $\widetilde{B}$ is then constructed via a backtracking line search.

A key feature of Algorithm~\ref{alg:iterative_reduce} is that the update magnitude is edge-adaptive rather than uniform. 
For each $e=(j\to i)\in\mathcal{S}$, the decrement is defined as
\[
\delta_{ij}
=
\min\!\left(
\sigma\left(
w_{ij}
+
d_f\frac{|\partial_e\kappa|}{\partial_{\max}}
\right),\;
w_{ij}
\right),~
\partial_{\max}=\max_{e\in\mathcal{S}}|\partial_e\kappa|,
\]
where $\sigma>0$ is a global scaling factor updated during backtracking. 

This composite rule couples two sources of information: 
(i) the current edge weight $w_{ij}$, which controls the admissible reduction scale, and 
(ii) the relative sensitivity magnitude $|\partial_e\kappa|/\partial_{\max}$, which prioritizes edges with stronger spectral impact. 
As a result, edges that are both heavy and highly influential receive larger updates.

The backtracking procedure adjusts the global factor $\sigma$ to ensure monotonic improvement of $\kappa$. 
If the trial graph $\widetilde{B}$ does not increase $\kappa$ or violates Assumptions~\ref{ass:A1}--\ref{ass:A3}, the algorithm contracts the step via
$\sigma \leftarrow \tfrac12 \sigma$,
and repeats the update. 
Therefore, while the backtracking mechanism is global, the effective step applied to each edge is jointly guided by the edge weight and the sensitivity. 

For comparison, we also consider a fixed-step greedy strategy. 
In that variant, once a set of candidate edges is selected, each chosen edge is updated with the same trial decrement $\sigma$, and backtracking is carried out only through the contraction $\sigma \leftarrow \tfrac12 \sigma$. 
In contrast to Algorithm~\ref{alg:iterative_reduce}, this fixed-step strategy does not distinguish between edges of different weights or different sensitivity magnitudes.

We test the two strategies on the original graph shown in Fig.~\ref{fig:derivative}(a). 
The experimental settings are as follows:
with step size $s=0.2$, derivative fraction $d_f=0.1$, batch size $k=10$, selection mode \texttt{choose\_mode} $\in\{$\texttt{topk, randomk, all}$\}$, weight threshold $\tau_w=10^{-4}$, numerical tolerance $\varepsilon=10^{-12}$, maximum number of iterations $T_{\max}=40$.

Fig.~\ref{fig:greedy} compares the evolution of $\kappa$ under different strategies. 
The three curves labeled ``Guided'' correspond to Algorithm~\ref{alg:iterative_reduce}, where edge updates are determined by a composite rule combining weight and sensitivity information, while the three ``Fixed'' curves use a uniform trial decrement.
The results suggest that the guided strategies tend to outperform the fixed-step variants. 
In particular, the \texttt{guided-all} and \texttt{guided-random} modes appear to perform better than the other strategies considered. 
Overall, incorporating sensitivity information leads to more effective improvement of $\kappa$.

\begin{figure}[htbp]
    \centering
    \includegraphics[width=\linewidth]{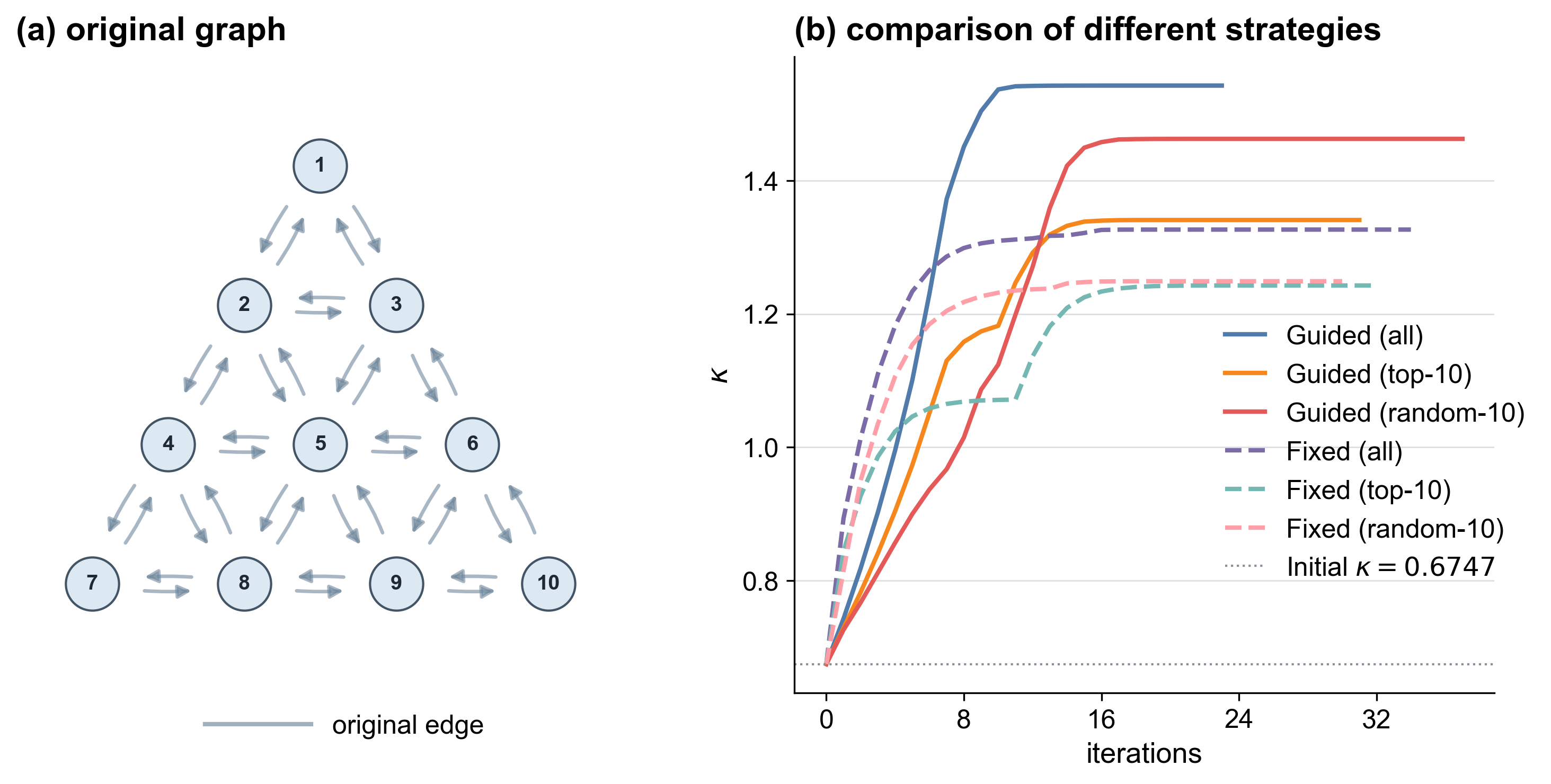}
    \caption{
    (a) Original directed weighted graph.
    (b) Comparison of different edge-weakening strategies.
The horizontal dotted line indicates the initial value of $\kappa$. 
The three ``Guided'' curves correspond to Algorithm~\ref{alg:iterative_reduce}, where the decrement of each selected edge is determined by a composite rule involving both the current edge weight and the normalized sensitivity magnitude, with backtracking on a common scaling factor. 
The three ``Fixed'' curves correspond to a baseline strategy in which all selected edges share the same trial decrement.
The guided strategies lead to larger improvements in $\kappa$ than the fixed-step variants.
    }
    \label{fig:greedy}
\end{figure}

The same framework can be adapted to edge strengthening by replacing the candidate set with $\mathcal{P}(B)\coloneqq \{e:\partial_e\kappa>0\}$, and can also be extended to mixed reweighting strategies that allow both positive and negative perturbations (see Section~\ref{sec:Discussion} for more details).

\begin{remark}[Computational complexity]
In the dense setting, computing the global spectral quantities $(\xi,\kappa,y)$ requires $O(n^3)$ operations per iteration. 
The main additional cost comes from evaluating $\partial_e\xi$ for candidate edges through the constrained linear system~\eqref{eq:ls}. 
If the corresponding bordered system is solved independently for each candidate edge, the cost is $O(n^3)$ per candidate. 
By factorizing the bordered coefficient matrix once per iteration and reusing it for all right-hand sides, each edge-wise solve costs $O(n^2)$. 
Thus, for $q$ candidate directed edges, the per-iteration cost is reduced from $O(qn^3)$ to $O(n^3+qn^2)$; in the worst case $q=O(n^2)$, this gives $O(n^4)$ per iteration and $O(T_{\max}n^4)$ over $T_{\max}$ iterations. 
Once $\partial_e\xi$ is obtained, the directed cut-energy term can be computed in $O(1)$, and the redistribution term can be computed in $O(n)$ per edge after precomputing $Cy$. 
Therefore, the evaluation of these remaining terms is lower order compared with the $O(n^2)$ edge-wise linear solve.
\end{remark}

\begin{remark}[$\gamma$-based variant]
The same sensitivity-guided strategy can be developed for $\gamma$ by replacing $\partial_e\kappa$ with $\partial_e\gamma$ throughout the algorithm. 
Since the $\gamma$-sensitivity does not involve stationary redistribution, the resulting implementation is simpler, although it lacks the additional structural information captured by the $\kappa$-based framework.
\end{remark}
\begin{algorithm}[t]
\caption{Sensitivity-guided discrete edge modification}
\label{alg:discrete_modify}
\begin{algorithmic}[1]
\Require Directed weighted graph $A$; operation type \texttt{op}
$\in\{\texttt{delete},\texttt{addneg}\}$; 
selection mode \texttt{choose\_mode} 
$\in\{\texttt{sortrandomk},\texttt{randomk},\texttt{all}\}$; 
batch size $k$; threshold $\tau_w$; tolerance $\varepsilon$; 
maximum iterations $T_{\max}$; negative weight $\omega_-<0$
\Ensure Modified graph $B$ with increased $\kappa$
\State $B\gets A$
\For{$t=1,\ldots,T_{\max}$}
    \State Compute $\partial_e\kappa$ for admissible candidate edges of $B$
    \If{\texttt{op} $=$ \texttt{delete}}
        \State $\mathcal P(B)\gets
        \{e=(j\to i): B_{ij}>\tau_w,\ \partial_e\kappa<-\varepsilon\}$
    \Else
        \State $\mathcal P(B)\gets
        \{e=(j\to i): |B_{ij}|\le \tau_w,\ \partial_e\kappa<-\varepsilon\}$
    \EndIf
    \If{$\mathcal P(B)=\varnothing$}
        \State \textbf{break}
    \EndIf
    \State Select $\mathcal S\subseteq\mathcal P(B)$ according to \texttt{choose\_mode}
    \State Randomly permute the edges in $\mathcal S$; accepted $\gets$ false
    \For{each $e=(j\to i)\in\mathcal S$ in the permuted order}
        \State $\widetilde B\gets B$
        \If{\texttt{op} $=$ \texttt{delete}}
            \State $\widetilde B_{ij}\gets 0$
        \Else
            \State $\widetilde B_{ij}\gets \omega_-$
        \EndIf
        \If{$\kappa(\widetilde B)>\kappa(B)+\varepsilon$ and 
        $\widetilde B$ satisfies Assumptions~\ref{ass:A1}--\ref{ass:A3}}
            \State $B\gets \widetilde B$; accepted $\gets$ true
            \State \textbf{break}
        \EndIf
    \EndFor
    \If{accepted $=$ false}
        \State \textbf{break}
    \EndIf
\EndFor
\State \Return $B$
\end{algorithmic}
\end{algorithm}
\subsection{Heuristic Variants: Edge Deletion and Negative-Edge Insertion}

The same sensitivity principle can be used to guide discrete edge modifications beyond continuous weight weakening. 
We consider two representative operations: deleting an existing edge and inserting a new directed edge with a prescribed negative weight $\omega_-<0$. 
For deletion, the admissible set is
\[
\mathcal P_{\rm del}(B)=
\{e=(j\to i): B_{ij}>\tau_w,\ \partial_e\kappa<-\varepsilon\},
\]
whereas for negative-edge insertion it is
\[
\mathcal P_{\rm add}(B)=
\{e=(j\to i): |B_{ij}|\le\tau_w,\ \partial_e\kappa<-\varepsilon\}.
\]

Given the admissible set $\mathcal P(B)$, a subset $\mathcal S$ is selected according to \texttt{choose\_mode}: 
\texttt{sortrandomk} selects the top-$k$ edges with largest $|\partial_e\kappa|$ and then tests them in random order; 
\texttt{randomk} selects $k$ edges uniformly at random; 
and \texttt{all} sets $\mathcal S=\mathcal P(B)$. 
Each candidate in $\mathcal S$ is then tested by applying the corresponding modification, namely setting $B_{ij}$ to $0$ for deletion or to $\omega_-$ for negative-edge insertion. 
The first candidate that yields $\kappa(\widetilde B)>\kappa(B)+\varepsilon$ while preserving Assumptions~\ref{ass:A1}--\ref{ass:A3} is accepted; otherwise the iteration terminates.

Because deletion and negative-edge insertion are not infinitesimal perturbations, $\partial_e\kappa$ is used only as a local ranking heuristic rather than as a guaranteed prediction of the resulting change in $\kappa$. 
Algorithm~\ref{alg:discrete_modify} summarizes the unified procedure.

\section{Numerical Results}\label{sec:Results}

In this section, we validate the proposed $\kappa$-sensitivity framework through numerical experiments on real directed networks. 
We consider three types of sensitivity-guided edge modifications: edge-weight weakening, edge deletion, and negative-edge insertion. 
The experiments are conducted on both first-order and second-order nonlinear consensus dynamics in order to assess the robustness and practical usefulness of the proposed perturbation-based design.

\subsection{Experimental Setup}

Our experiments are based on the Email-Eu-core network~\cite{leskovec2007graph,yin2017local}, a real-world directed network containing $1005$ nodes and $25571$ directed edges. 
The raw edge list is first relabeled so that node identifiers form consecutive indices, and all self-loops are removed. 
We then construct a directed unweighted graph by assigning unit weight to all retained edges. 
For each target size, we sample a subset of high-degree nodes and extract the corresponding induced subgraph. 
To ensure consistency with our theoretical setting, we retain its largest strongly connected component (SCC), 
so that Assumptions~\ref{ass:A1} and~\ref{ass:A2} hold under the nonnegative-weight setting. 
We further verify numerically that Assumption~\ref{ass:A3} is satisfied. 
If necessary, we perform a mild additional cleanup by iteratively removing problematic nodes and retaining the largest strongly connected component, until all assumptions are met.

Using this procedure, we obtain two representative subgraphs with $100$ and $199$ nodes, containing $3823$ and $8582$ directed edges, respectively. 
These subgraphs are represented by adjacency matrices $A$, where $A_{ij}=1$ indicates a directed edge from node $j$ to node $i$. 
The smaller subgraph ($100$ nodes) is used for second-order nonlinear dynamics, while the larger one ($199$ nodes) is used for first-order nonlinear dynamics. 
Both choices balance computational tractability and structural complexity, enabling clear observation of the transition from non-synchronization to synchronization under sensitivity-guided modifications.

We consider three types of edge-level modifications:
(i) weakening  existing edges (Algorithm~\ref{alg:iterative_reduce}, \texttt{topk} mode);
(ii) removing existing edges (Algorithm~\ref{alg:discrete_modify}, \texttt{delete} operation and \texttt{sortrandomk} mode);
(iii)inserting negative edges  (Algorithm~\ref{alg:discrete_modify}, \texttt{addneg} operation, \texttt{sortrandomk} mode, and negative weight $\omega_-=-1.0$).

We evaluate synchronization behavior for both first-order and second-order nonlinear consensus systems. 
Initial conditions are drawn from zero-mean Gaussian distributions with small-amplitude scaling. 
Specifically, for first-order systems, we set $x_i(0)=0.02\,\eta_i$,
and for second-order systems, $x_i(0)=0.02\,\eta_i$, $v_i(0)=0.02\,\zeta_i$, where $\eta_i,\zeta_i \sim \mathcal{N}(0,1)$.
In all experiments, a single realization of the initial condition, generated using a fixed random seed, is used consistently across all modification steps to ensure a controlled and fair comparison.

All dynamical systems are simulated using an adaptive Runge--Kutta method (RK45) implemented via \texttt{solve\_ivp}, with tolerances $\mathrm{rtol}=10^{-8}$ and $\mathrm{atol}=10^{-10}$. 
Time trajectories are evaluated over a sufficiently long horizon, with a nominal sampling step $\Delta t=0.05$. 
The total simulation time is set to $300$ for first-order systems and $500$ for second-order systems.

We use the mean-square deviation from the network average to quantify synchronization performance. 
For first-order systems,
\begin{align*}
e(t)=\frac{1}{n}\sum_{i=1}^n\bigl(x_i(t)-\bar{x}(t)\bigr)^2,
\qquad
\bar{x}(t)=\frac{1}{n}\sum_{i=1}^n x_i(t),
\end{align*}
while for second-order systems,
\begin{align*}
e(t)=
\frac{1}{n}\sum_{i=1}^n\bigl(x_i(t)-\bar{x}(t)\bigr)^2
+
\frac{1}{n}\sum_{i=1}^n\bigl(v_i(t)-\bar{v}(t)\bigr)^2.
\end{align*}

We further consider a time-averaged error
\begin{align*}
E(T)=\frac{1}{T_w}\int_T^{T+T_w} e(t)\,dt,
\end{align*}
to capture long-time behavior,
where $T_w$ denotes the observation window. 
In our experiments, we set $(T,T_w)=(240,60)$ for first-order systems and $(T,T_w)=(400,100)$ for second-order systems, so that the initial transient phase is excluded. 
The quantity $\log E$ is reported in the figures for visualization.

\begin{figure*}[t]
\centering
\includegraphics[width=\textwidth]{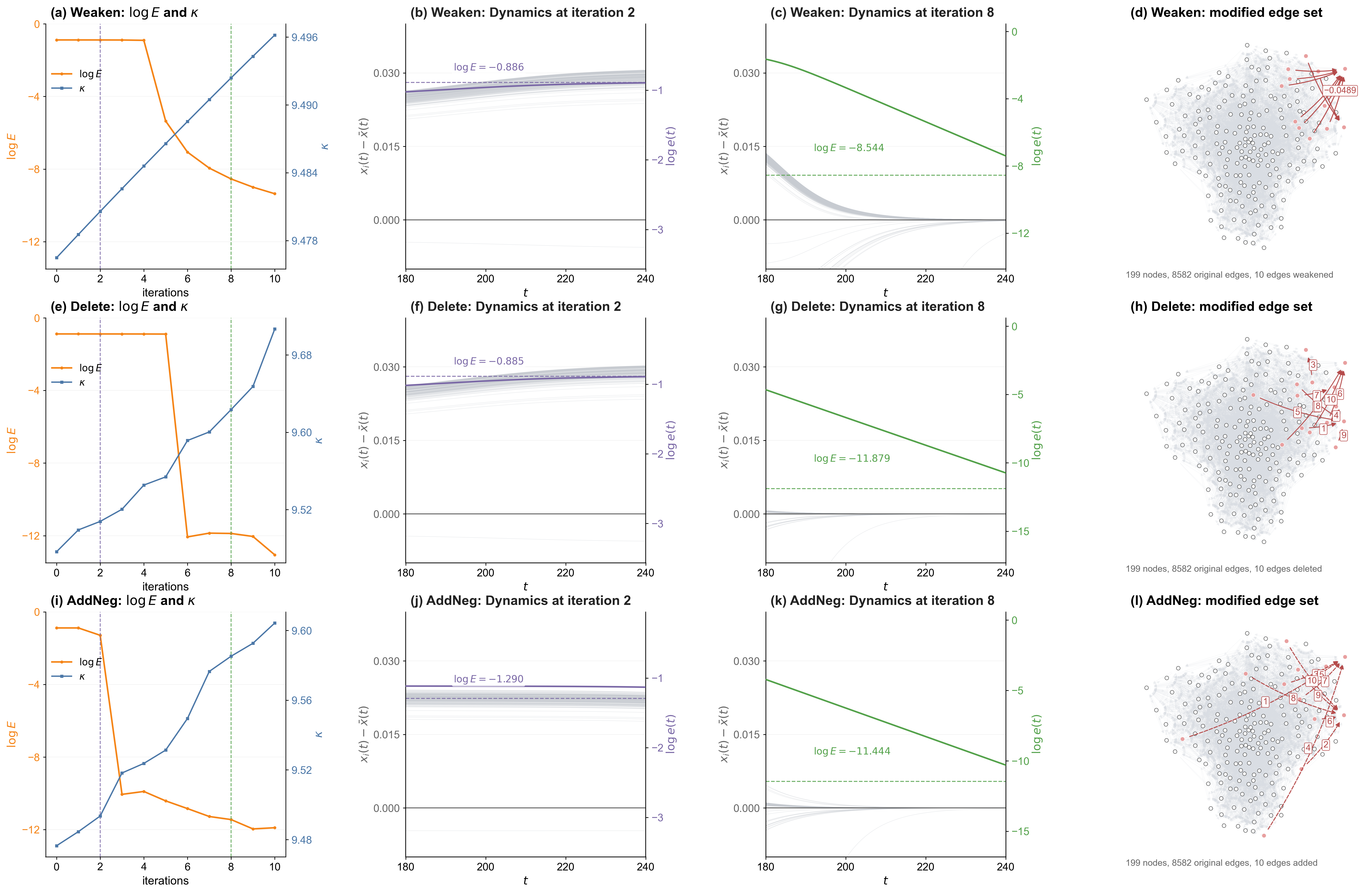}
\caption{Sensitivity-guided modification for first-order nonlinear consensus dynamics on the Email-Eu-core subgraph ($n=199$).
The system is governed by $\dot{x}=0.1\sin(x)-cLx$ with $c=1.65\times10^{-3}$. 
Each row corresponds to one modification strategy: 
(a)--(d) edge-weight weakening, 
(e)--(h) edge deletion, and 
(i)--(l) negative-edge insertion.
In each row, 
the first column shows the evolution of $\log E$ (left axis) and $\kappa$ (right axis) versus iteration step. 
The second and third columns display the state trajectories $x_i(t)-\bar{x}(t)$ at an early step (iteration 2) and a later step (iteration 8), respectively, together with the corresponding $\log e(t)$. 
The fourth column visualizes the modified edges on the network, where highlighted edges indicate weakened, removed, or newly inserted negative connections.
Across all strategies, $\kappa$ increases monotonically with the modification steps, while the synchronization error decreases by several orders of magnitude. 
In the reported example, the dynamical trajectories exhibit a transition from a nonsynchronized regime to a synchronized state.}
\label{fig:first_nonlinear}
\end{figure*}

\begin{figure*}[t]
\centering
\includegraphics[width=\textwidth]{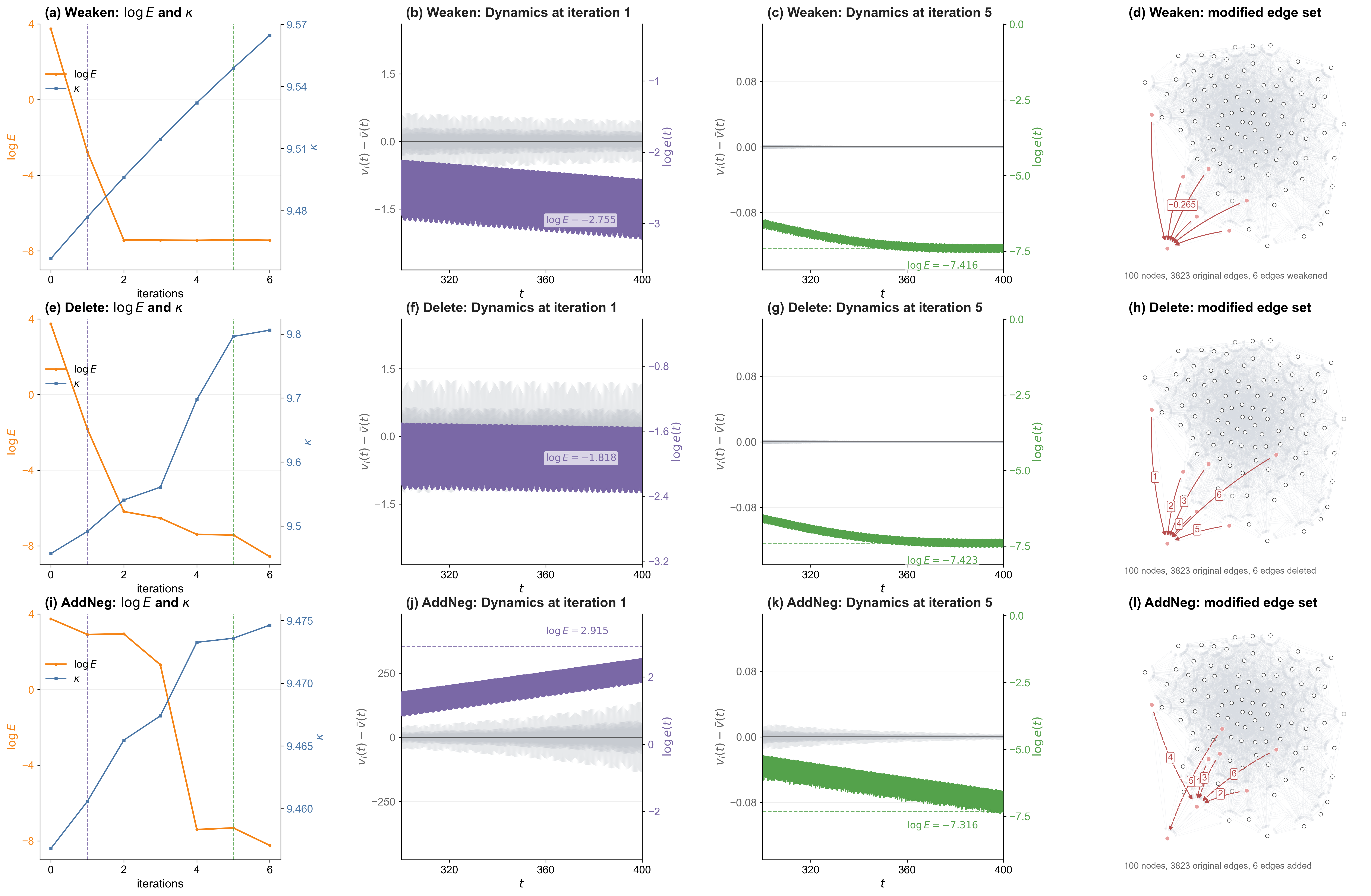}
\caption{Sensitivity-guided modification for second-order nonlinear consensus dynamics on the Email-Eu-core subgraph ($n=100$).
The system is governed by
$\dot{x}=v$,
$\dot{v}=0.1\tanh(x)+0.1\tanh(v)-\alpha Lx-\beta Lv$
with $(\alpha,\beta)=(40,0.0158)$. 
Each row corresponds to one modification strategy: 
(a)--(d) edge-weight weakening, 
(e)--(h) edge deletion, and 
(i)--(l) negative-edge insertion.
In each row, 
the first column shows the evolution of $\log E$ (left axis) and $\kappa$ (right axis) versus iteration step. 
The second and third columns display the early-step and late-step trajectories of the second-order dynamics (velocity deviations), together with $\log e(t)$. 
The fourth column visualizes the modified edges, where highlighted edges indicate weakened, removed, or newly inserted negative connections.
Across all strategies, $\kappa$ increases monotonically, while the synchronization error decreases significantly. 
In the reported higher-order nonlinear example, the dynamical trajectories show a transition from a nonsynchronized regime to a synchronized state.}
\label{fig:second_nonlinear}
\end{figure*}

\subsection{First-Order Nonlinear Systems}

We first consider first-order nonlinear consensus dynamics of the form
$\dot{x}=f(x)-cLx$,
where $f(x)=0.1\sin(x)$, $c=1.65\times10^{-3}$.

The weight threshold and numerical tolerance are set to $\tau_w=10^{-4}$ and $\varepsilon=10^{-12}$, respectively. 
For the edge-weight weakening algorithm, we use batch size $k=10$, step size  $s=0.005$, and derivative fraction $d_f=0$.
The small value of $s$ ensures that the selected edge set remains stable across iterations. 
Moreover, by setting $d_f = 0$, the update magnitude becomes independent of the sensitivity, resulting in identical weight reductions for all selected edges. 
For edge deletion and negative-edge insertion, we use a larger batch size $k=35$. 
All methods are run for $T_{\max}=10$ iterations.

Fig.~\ref{fig:first_nonlinear} summarizes the effect of sensitivity-guided edge modifications on synchronization. 
Across all three intervention types—edge-weight weakening, edge deletion, and negative-edge insertion—we observe a consistent monotonic increase in $\kappa$, accompanied by a pronounced decrease in the synchronization error $E$. 
In particular, the initial network exhibits a non-synchronized regime with relatively large steady-state error, while in the reported example, the modified networks enter a synchronized regime after only a few steps.

The dynamical trajectories provide a direct illustration of this transition. 
At early stages (e.g., iteration 2), node states remain dispersed and the error $e(t)$ stabilizes at a relatively high level. 
At later stages (e.g., iteration 8), all trajectories collapse toward a common value, and $e(t)$ rapidly decays, indicating synchronization. 

We emphasize that these improvements are achieved through highly sparse modifications: in all cases, only a small number of edges (at most $10$) are altered. 
Despite their localized nature, these changes induce a global transition in the system behavior, highlighting the nonlocal influence of directed interactions.

Notably, the improvement is achieved through counterintuitive modifications: rather than strengthening the network uniformly, the method selectively weakens or reverses a few interactions, leading to faster convergence. 
This phenomenon reflects the combined effect of the directed cut-energy term and the stationary redistribution effect.

\subsection{Second-Order Nonlinear Systems}

We next consider second-order nonlinear consensus dynamics of the form
$\dot{x}=v$, $\dot{v}=f(x,v)-\alpha Lx-\beta Lv$,
where $f(x,v)=0.1\tanh(x)+0.1\tanh(v)$, $(\alpha, \beta) = (40, 0.0158)$.

The weight threshold and numerical tolerance are set to $\tau_w=10^{-4}$ and $\varepsilon=10^{-12}$, respectively. 
For the edge-weight weakening algorithm, we use batch size $k=6$, step size $s=0.05$, and derivative fraction $d_f=0$. 
For edge deletion and negative-edge insertion, we use batch size $k=4$. 
All methods are run for $T_{\max}=6$ iterations.

The results are shown in Fig.~\ref{fig:second_nonlinear}. 
All three sensitivity-guided strategies remain applicable to second-order systems with coupled position and velocity variables, yielding consistent improvements in $\kappa$ and reductions in the synchronization error.

At early stages (e.g., iteration 1), the trajectories remain widely dispersed, and the error $e(t)$ stabilizes at a relatively high level, indicating a non-synchronized regime. 
As the modification steps proceed, the system undergoes a sharp transition in the reported example: at later stages (e.g., iteration 4 or 5), both position and velocity components converge rapidly, and $e(t)$ decays. 
This behavior suggests that, for the specific second-order system considered here, increasing $\kappa$ effectively enhances synchronization even in the presence of higher-order dynamics and nonlinear coupling.

Despite the increased dynamical complexity, the required modifications remain sparse. 
In all cases, only a small number of edges (at most $6$) are modified, yet the global behavior of the system is fundamentally altered. 
This again highlights the nonlocal influence of directed interactions and demonstrates that carefully selected local perturbations can control macroscopic synchronization properties.

\section{Discussion}\label{sec:Discussion}
The preceding results show that edge weakening, edge deletion, and negative-edge insertion can increase the generalized algebraic connectivity $\kappa$ in directed networks. 
However, there is an upper bound on the largest value of $\kappa$ that can be attained solely through edge-weight reduction (see Proposition~\ref{PROP:kappa-upper} in the Appendix):
\begin{align*}
    \kappa \le \frac{1}{n-1}\sum_{j\to i \in E} a_{ij}.
\end{align*} 
Therefore, if one aims to further enlarge the parameter region suggested by directed second-order consensus theory~\cite{yu2009second}, it is natural to consider the opposite operation: increasing selected edge weights under a fixed budget. 

This leads to a budget-constrained edge-strengthening problem: given a positively weighted directed graph and a total additional weight budget $\mathcal B>0$, how should the budget be allocated in order to increase $\kappa$ most effectively? 
Under the perturbation convention~\eqref{eq:Ae}, positive perturbations correspond to edge-weight increases, and thus edges with $\partial_e\kappa>0$ are favorable candidates for strengthening. 
Algorithm~\ref{alg:budget_strengthening} implements this idea by repeatedly selecting sensitivity-positive edges and accepting only updates that lead to an actual increase in $\kappa$.

\begin{algorithm}[h]
\caption{Budget-constrained sensitivity-guided edge strengthening (OURS-FIXED/GUIDED)}
\label{alg:budget_strengthening}
\begin{algorithmic}[1]
\Require Directed weighted graph $A$; total budget $\mathcal{B}>0$;
         step size $s>0$; 
         batch size $k\geq 1$; selection mode \texttt{choose\_mode} $\in \{\texttt{topk, randomk, all}\}$;
         numerical tolerance $\varepsilon>0$;
         allocation mode \texttt{alloc} $\in \{\texttt{FIXED, GUIDED}\}$
\Ensure Modified graph $B$ with increased $\kappa$ under budget $\mathcal{B}$

\State $B \gets A$, \quad $\mathcal{B}_{\mathrm{left}} \gets \mathcal{B}$
\While{true}
    \If{$\mathcal{B}_{\mathrm{left}} \leq \varepsilon$}
        \State \textbf{break}
    \EndIf

    \State Compute sensitivities $\partial_e\kappa$ for all existing edges of $B$
    \State Define the candidate strengthening set
    \begin{align*}
        \mathcal{P}(B) \coloneqq \{e=(j\to i) : \partial_e\kappa > 0\}
    \end{align*}
    \If{$\mathcal{P}(B) = \varnothing$}
        \State \textbf{break}
    \EndIf
    \State Select $\mathcal{S} \subseteq \mathcal{P}(B)$ according to \texttt{choose\_mode}

    \State $\alpha \gets \min(s,\mathcal{B}_{\mathrm{left}})$, \quad accepted $\gets$ false
    \While{$\alpha > \varepsilon$ \textbf{ and } accepted $=$ false}
        \State Construct $\widetilde{B} \gets B$
        \State Compute edge increments according to \texttt{alloc}:
        \[
        \delta_e \gets
        \begin{cases}
        \dfrac{\alpha}{|\mathcal{S}|},
        & \texttt{FIXED}, \\[8pt]
        \alpha\,\dfrac{\partial_e\kappa}
        {\displaystyle\sum_{e'\in\mathcal{S}}\partial_{e'}\kappa},
        & \texttt{GUIDED}.
        \end{cases}
        \]
        \For{each $e=(j\to i)\in\mathcal{S}$}
            \State $\widetilde{B}_{ij} \gets B_{ij} + \delta_e$
        \EndFor
        \If{$\kappa(\widetilde{B}) > \kappa(B)$}
            \State $\Delta \gets \sum_{e\in\mathcal{S}} \delta_e$
            \Comment{actual added weight}
            \State $B \gets \widetilde{B}$, \quad
                   $\mathcal{B}_{\mathrm{left}} \gets \mathcal{B}_{\mathrm{left}} - \Delta$
            \State accepted $\gets$ true
        \Else
            \State $\alpha \gets \frac{1}{2}\alpha$
        \EndIf
    \EndWhile

    \If{accepted $=$ false}
        \State \textbf{break}
    \EndIf
\EndWhile
\State \Return $B$
\end{algorithmic}
\end{algorithm}

We compare two sensitivity-based allocation rules. 
OURS-FIXED selects the top-$k$ sensitivity-positive edges and distributes the trial increment uniformly among them, whereas OURS-GUIDED allocates the same trial increment proportionally to the sensitivity magnitudes. 
Both methods recompute sensitivities after each accepted update and use backtracking to ensure monotonic improvement of $\kappa$. 
As baselines, we include directed edge betweenness centrality (Dir-EBC)~\cite{girvan2002community}, directed degree-assortativity score (Dir-DAC)~\cite{newman2002assortative}, Random, and Uniform. 
All methods are evaluated under the same total budget and the same acceptance criterion.

For the edge-based structural baselines, Dir-EBC assigns each directed edge $e$ the score
\[
\mathrm{EBC}(e)
=
\sum_{s\neq t}\frac{\sigma_{st}(e)}{\sigma_{st}},
\]
where $\sigma_{st}$ is the number of shortest directed paths from $s$ to $t$, and $\sigma_{st}(e)$ is the number of such paths passing through $e$. 
For weighted graphs, edge lengths are taken as the inverse of the corresponding weights. 
Dir-DAC assigns each edge $j\to i$ the score
\[
\mathrm{DAC}(j\to i)
=
(d_j^{\mathrm{out}}-\bar d^{\mathrm{out}})
(d_i^{\mathrm{in}}-\bar d^{\mathrm{in}}),
\]
where $d_j^{\mathrm{out}}$ and $d_i^{\mathrm{in}}$ are the weighted out-degree and weighted in-degree, respectively, and $\bar d^{\mathrm{out}}$ and $\bar d^{\mathrm{in}}$ are their graph averages. 
For Random, existing edges are assigned i.i.d.\ random scores and ranked accordingly. 
For Uniform, no ranking is used, and the budget is distributed evenly over all existing edges.

All methods are evaluated under the same total budget $\mathcal{B}=10$. 
At each accepted step, the remaining budget is updated and the corresponding value of $\kappa$ is recorded. 
The candidate pool size is fixed at $k=10$. 
OURS-GUIDED and OURS-FIXED recompute edge sensitivities at each iteration, select the top-$k$ improving edges, and update them simultaneously through backtracking. 
Dir-EBC, Dir-DAC, and Random use a \texttt{sortrandomk} scheme: the top-$k$ candidates according to the corresponding score are selected, randomly permuted, and examined sequentially, and only the first edge yielding a strict increase in $\kappa$ is updated.

Fig.~\ref{fig:budget_compare_kappa} shows the results on directed Erd\H{o}s--R\'enyi and directed small-world networks. 
In both cases, the sensitivity-based methods outperform the structural and random baselines, indicating that first-order $\kappa$-sensitivity provides useful information not only for identifying favorable edges, but also for allocating a strengthening budget. 
Moreover, OURS-GUIDED consistently improves over OURS-FIXED, showing that proportional allocation according to sensitivity magnitudes is more effective than uniform allocation over the same selected edge set. 
These results demonstrate that the proposed framework is not limited to reducing coupling; it can also guide the efficient use of additional coupling resources.

\begin{figure*}[t]
    \centering
    \includegraphics[width=\linewidth]{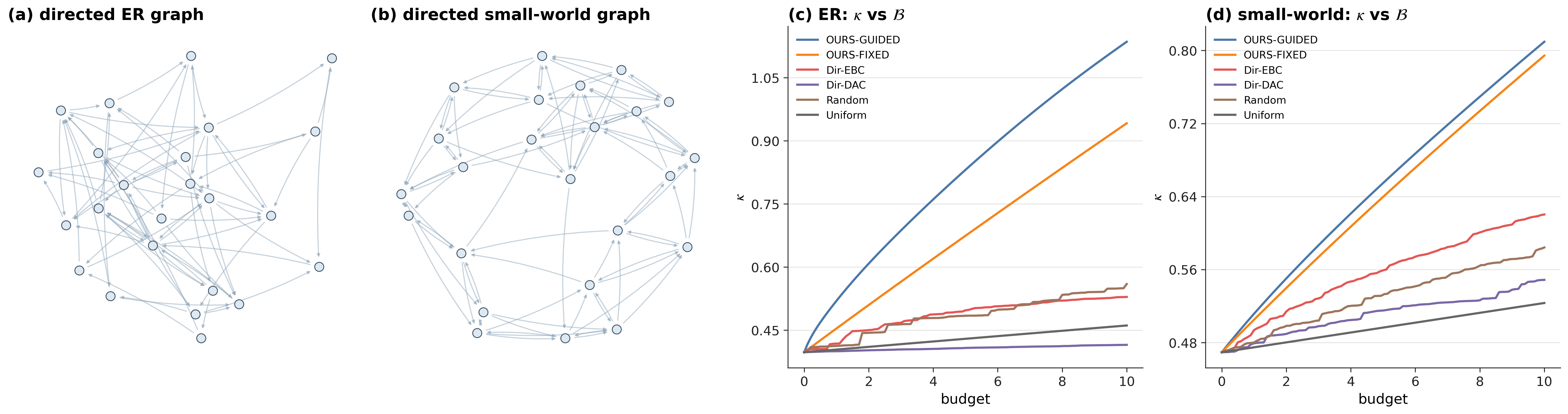}
    \caption{
    Budget-constrained edge strengthening on directed networks.
    (a) A directed Erd\H{o}s--R\'enyi graph with $n=24$ nodes and connection probability $p=0.14$, where existing edge weights are independently sampled from $[0.6,1.4]$.
    (b) A directed small-world graph constructed from a directed ring lattice with nearest-neighbor connections in both directions, augmented with second- and third-neighbor links and random rewiring with probability $0.20$, with edge weights sampled from $[0.6,1.4]$.
    In both cases, the graph is required to be strongly connected.
    (c)--(d) Evolution of the generalized algebraic connectivity $\kappa$ as a function of the consumed budget $\mathcal{B}$.
   At each iteration, OURS-GUIDED and OURS-FIXED recompute sensitivities and update the \texttt{topk} edges via backtracking. 
OURS-FIXED uses uniform allocation, whereas OURS-GUIDED allocates updates proportionally to the sensitivity magnitudes. 
The structural baselines (Dir-EBC, Dir-DAC) and Random follow a \texttt{sortrandomk} rule, in which a subset of candidates is selected, randomly permuted, and examined sequentially, accepting the first edge that yields a strict increase in $\kappa$. 
Uniform distributes the budget evenly over all edges.
The sensitivity-based methods consistently achieve the largest improvement in $\kappa$, and the proportional allocation in OURS-GUIDED further improves performance over the uniform allocation in OURS-FIXED.
    }
    \label{fig:budget_compare_kappa}
\end{figure*}

\section{Conclusion}\label{sec:Conclusion}

In this paper, we developed a perturbation-based $\kappa$-sensitivity framework for synchronization analysis and optimization in directed weighted networks. 
Our main theoretical contribution is an explicit first-order perturbation formula for the generalized algebraic connectivity $\kappa$, together with a decomposition into a directed cut-energy term and a stationary redistribution term. 
This decomposition reveals that, in directed networks, the synchronization role of an edge is determined not only by local edgewise disagreement, but also by the global redistribution of stationary mass induced by perturbations.

Based on this spectral sensitivity analysis, we proposed sensitivity-guided strategies for edge weakening, edge deletion, negative-edge insertion, and budget-constrained weight strengthening. 
The numerical results show that, in the reported experiments, these strategies outperform standard structural baselines on synthetic and real directed networks. 
In particular, the experiments exhibit the counterintuitive phenomenon emphasized in this paper: weakening a small set of carefully selected edges can increase $\kappa$, accelerate convergence, and, in some nonlinear examples, induce a transition from nonsynchronization to synchronization.

The proposed framework therefore provides a unified local spectral explanation of how less coupling can improve synchronization-related spectral performance, and helps explain why, in some directed nonlinear systems, this may also be accompanied by a transition to synchronization. 
At the same time, the same sensitivity principle can also be used in the opposite direction, namely to allocate additional coupling resources more effectively when strengthening interactions is desired. 
This makes the framework broadly applicable to synchronization enhancement, network design, and control of directed networked dynamical systems, while keeping clear that the present theory is first-order and local in nature.

Several directions remain open for future work. 
On the theoretical side, it would be interesting to extend the perturbation analysis beyond the first-order regime and to study more systematically the connection between $\kappa$ and synchronization thresholds in nonlinear higher-order systems. 
On the algorithmic side, one may explore more scalable optimization strategies for large networks and more structured perturbation models, such as node-level interventions or constrained signed modifications. 
These questions may further broaden the applicability of perturbation-based spectral sensitivity methods in complex network design and control.


\appendix[Analyticity and Proofs of the Main Results]

To establish the differentiability of $\kappa$, we use analytic perturbation theory for Hermitian generalized eigenvalue problems~\cite{andrew1998computation}. 
Since the generalized eigenvalue formulation for $\kappa$ involves the stationary distribution $\xi$, we first prove the analyticity of $\xi(\epsilon)$ under the perturbation model~\eqref{eq:Ae}.

\begin{theorem}\label{thm:analytic}
Let $P(\rho)$ and $Q(\rho)$ be matrix-valued functions analytic in $\rho$ in a neighborhood of $\rho_0$, and assume that for each $\rho$ in this neighborhood,
\[
P(\rho)=P(\rho)^*,
\qquad
Q(\rho)=Q(\rho)^*,
\]
with $Q(\rho)$ positive definite. 
Consider the generalized Hermitian eigenvalue problem
\[
P(\rho)x(\rho)=\lambda(\rho)\,Q(\rho)x(\rho).
\]
Then the eigenvalues admit analytic parameterizations in a neighborhood of $\rho_0$. 
Moreover, if $\lambda(\rho_0)$ is simple, then there exist an analytic eigenvalue branch $\lambda(\rho)$ and an associated analytic eigenvector branch $x(\rho)$, unique up to analytic normalization.
\end{theorem}

\begin{lemma}\label{lem:xi-analytic}
Under Assumptions~\ref{ass:A1} and \ref{ass:A2}, there exists $\varepsilon_0>0$ such that, for all $|\epsilon|<\varepsilon_0$, the normalized left zero-eigenvector $\xi(\epsilon)$ defined by
\[
\xi(\epsilon)^\top L(\epsilon)=0,
\qquad
\mathbf{1}^\top \xi(\epsilon)=1
\]
is uniquely determined and analytic in $\epsilon$. 
Consequently, $\Xi(\epsilon)=\mathrm{diag}(\xi(\epsilon))$ and $\Xi(\epsilon)^{-1/2}$ are analytic in $\epsilon$ for $|\epsilon|<\varepsilon_0$.
\end{lemma}

\begin{proof}
By~\eqref{eq:Ae}, the perturbed adjacency matrix $A(\epsilon)$ is analytic in $\epsilon$, and hence so is the Laplacian matrix $L(\epsilon)$. 
Moreover, for every $\epsilon$, $L(\epsilon)\mathbf{1}=0$,
so $0$ remains an eigenvalue of $L(\epsilon)$.

Consider the bordered matrix
\[
H(\epsilon)\coloneqq 
\begin{bmatrix}
L(\epsilon)^\top & \mathbf{1}\\
\mathbf{1}^\top & 0
\end{bmatrix}.
\]
We first show that $H(0)$ is invertible. 
Suppose
\[
H(0)
\begin{bmatrix}
u\\ \mu
\end{bmatrix}
=
0.
\]
Then$L^\top u+\mu \mathbf{1}=0$, $\mathbf{1}^\top u=0$.
Multiplying the first equation on the left by $\mathbf{1}^\top$ and using $L\mathbf{1}=0$, we obtain $n\mu=0$,
hence $\mu=0$. Therefore $L^\top u=0$. Since $0$ is a simple eigenvalue of $L$ by Assumption~\ref{ass:A1}, the nullspace of $L^\top$ is one-dimensional and spanned by $\xi$. Thus $u=c\,\xi$ for some scalar $c$. The constraint $\mathbf{1}^\top u=0$, together with $\mathbf{1}^\top \xi=1$, yields $c=0$, so $u=0$. Hence $H(0)$ is invertible.

Since $H(\epsilon)$ depends analytically on $\epsilon$ and $H(0)$ is invertible, there exists $\varepsilon_0>0$ such that $H(\epsilon)$ remains invertible for all $|\epsilon|<\varepsilon_0$, and $H(\epsilon)^{-1}$ is analytic in $\epsilon$. 
Define
\[
\begin{bmatrix}
\xi(\epsilon)\\ \mu(\epsilon)
\end{bmatrix}
\coloneqq 
H(\epsilon)^{-1}
\begin{bmatrix}
0\\ 1
\end{bmatrix}.
\]
Then $\xi(\epsilon)$ and $\mu(\epsilon)$ are analytic in $\epsilon$ and satisfy
$L(\epsilon)^\top \xi(\epsilon)+\mu(\epsilon)\mathbf{1}=0$, $\mathbf{1}^\top \xi(\epsilon)=1$.
Multiplying the first equation on the left by $\mathbf{1}^\top$ gives
$n\mu(\epsilon)=0$,
hence $\mu(\epsilon)=0$ and therefore
$L(\epsilon)^\top \xi(\epsilon)=0$, $\mathbf{1}^\top \xi(\epsilon)=1$.
This proves the analyticity and uniqueness of $\xi(\epsilon)$.

Finally, since $\xi(0)>0$ by Assumption~\ref{ass:A2}, continuity implies that $\xi_i(\epsilon)>0$ for all $i$ and all sufficiently small $|\epsilon|$. Therefore $\Xi(\epsilon)$ and $\Xi(\epsilon)^{-1/2}$ are well defined and analytic in $\epsilon$.
\end{proof}

Based on Lemma~\ref{lem:xi-analytic}, the matrices
\[
\widehat{\Xi(\epsilon)L(\epsilon)}
=
\frac{\Xi(\epsilon)L(\epsilon)+L(\epsilon)^\top \Xi(\epsilon)}{2}
\qquad\text{and}\qquad
\Xi(\epsilon)
\]
are analytic in $\epsilon$. 
Moreover, $\widehat{\Xi(\epsilon)L(\epsilon)}$ is symmetric and $\Xi(\epsilon)$ is positive definite for sufficiently small $|\epsilon|$. 
Under Assumption~\ref{ass:A3}, $\kappa(0)$ is a simple generalized eigenvalue of the pair $\bigl(\widehat{\Xi L},\,\Xi\bigr)$.
Hence, by Theorem~\ref{thm:analytic}, there exist analytic branches $\kappa(\epsilon)$ and $y(\epsilon)$ such that
\begin{align}\label{eq:gen-eig}
\widehat{\Xi(\epsilon)L(\epsilon)}\,y(\epsilon)
=
\kappa(\epsilon)\,\Xi(\epsilon)\,y(\epsilon),
\quad
y(\epsilon)^\top \Xi(\epsilon)y(\epsilon)=1.
\end{align}
Since $M(\epsilon)=\Xi(\epsilon)^{-1/2}\widehat{\Xi(\epsilon)L(\epsilon)}\,\Xi(\epsilon)^{-1/2}$, this generalized eigenvalue problem is equivalent to the standard eigenvalue problem $M(\epsilon)v(\epsilon)=\kappa(\epsilon)v(\epsilon)$, $v(\epsilon)=\Xi(\epsilon)^{1/2}y(\epsilon)$,
with $\|v(\epsilon)\|_2=1$.

We are now ready to prove Theorem~\ref{thm:kappa}.

\begin{proof}[Proof of Theorem~\ref{thm:kappa}]
Let $P(\epsilon)\coloneqq \widehat{\Xi(\epsilon)L(\epsilon)}$ and $Q(\epsilon)\coloneqq \Xi(\epsilon)$.
Then~\eqref{eq:gen-eig} reads
\[
P(\epsilon)y(\epsilon)=\kappa(\epsilon)Q(\epsilon)y(\epsilon),
\qquad
y(\epsilon)^\top Q(\epsilon)y(\epsilon)=1.
\]
Differentiating the generalized eigenvalue equation with respect to $\epsilon$ gives
$P' y + P y'
=
\kappa' Q y + \kappa Q' y + \kappa Q y'$,
where all quantities are evaluated at $\epsilon=0$. 
Multiplying on the left by $y^\top$ and using the symmetry of $P$ and $Q$, together with $Py=\kappa Qy$, $y^\top Q y=1$,
we obtain the standard first-order formula
\begin{align}\label{eq:kappa-derivative-general}
\kappa'
=
y^\top(P'-\kappa Q')y.
\end{align}

We now compute $P'$ and $Q'$. Since $Q'=\Xi'$ and
\[
P'
=
\frac{1}{2}\bigl(\Xi' L+\Xi L'+L'^T\Xi+L^T\Xi'\bigr),
\]
substituting into~\eqref{eq:kappa-derivative-general} yields
\begin{align*}
\kappa'
&=
\frac12\,y^\top
\bigl(
\Xi' L+\Xi L'+L'^T\Xi+L^T\Xi'-2\kappa\Xi'
\bigr)y.
\end{align*}
Because all quantities are real scalars,
$y^\top L'^T\Xi y = y^\top \Xi L' y$, $y^\top L^T\Xi' y = y^\top \Xi' L y$.
Therefore
\begin{align}\label{eq:kappa-derivative-simplified}
\kappa'
=
y^\top \Xi L' y
+
y^\top \Xi'(L-\kappa I)y.
\end{align}

It remains to evaluate the first term. 
By the perturbation model~\eqref{eq:Ae}, one has
\[
A'=\sum_{j\to i\in F}E_{ij},
\qquad
D'=\sum_{j\to i\in F}E_{ii},
\]
and hence $L'
=
D'-A'$.
Thus
\begin{align*}
y^\top \Xi L' y
=
\sum_{j\to i\in F}
y^\top \Xi(E_{ii}-E_{ij})y
=
\sum_{j\to i\in F}
\xi_i\,y_i\,(y_i-y_j).
\end{align*}

Next, we observe that the equation defining $\xi'=\partial_F\xi$ is linear in the perturbation $L'$, and hence linear in the edge set $F$. 
Indeed, differentiating $\xi(\epsilon)^\top L(\epsilon)=0$, $\mathbf{1}^\top\xi(\epsilon)=1$,
gives $L^\top \partial_F \xi = -L'^T \xi$, $\mathbf{1}^\top \partial_F \xi = 0$.
Since $L'$ is the sum of the individual edge contributions, it follows that
\[
\partial_F \xi = \sum_{j\to i\in F}\partial_e \xi,
\qquad
\partial_F \Xi=\sum_{j\to i\in F}\partial_e \Xi.
\]
Combining this with~\eqref{eq:kappa-derivative-simplified}, we obtain
\begin{align*}
\partial_F \kappa
=
\underbrace{
\sum_{j\to i\in F}
\xi_i\,y_i\,(y_i-y_j)
}_{\text{Directed Cut Energy}}
+
\underbrace{
y^\top
\partial_F \Xi
(L-\kappa I)y
}_{\text{Stationary Redistribution Effect}}.
\end{align*}
This completes the proof.
\end{proof}

\begin{proof}[Proof of Proposition~\ref{prop:CD}]
Recall that $M=\Xi^{-1/2}\widehat{\Xi L}\,\Xi^{-1/2}$,
and let $v$ be the normalized eigenvector of $M$ associated with $\kappa$, i.e.,
$Mv=\kappa v$.
Set $y=\Xi^{-1/2}v$.
Then $v=\Xi^{1/2}y$, and multiplying the eigenvalue equation by $\Xi^{1/2}$ gives
\begin{align}\label{eq:aux-gen}
\widehat{\Xi L}\,y=\kappa\,\Xi y.
\end{align}
Using $\widehat{\Xi L}=(\Xi L+L^\top\Xi)/2$,
we obtain from~\eqref{eq:aux-gen}
\[
\Xi L\,y
=
\kappa\,\Xi y+\frac12(\Xi L-L^\top\Xi)y.
\]
Recalling the definition
$C\coloneqq \Xi L-L^\top\Xi$,
we get
$(L-\kappa I)y=\Xi^{-1}Cy/2$.
Now define $D\coloneqq  (\partial_F\Xi)\Xi^{-1}$.
Then the stationary redistribution term satisfies
\begin{align*}
R
\coloneqq 
y^\top(\partial_F\Xi)(L-\kappa I)y
=
\frac12\,y^\top(\partial_F\Xi)\Xi^{-1}Cy
=
\frac12\,y^\top D C y.
\end{align*}
This proves~\eqref{eq:RDC}.
\end{proof}

\begin{proposition}\label{PROP:kappa-upper}
Under Assumptions~\ref{ass:A1}--\ref{ass:A3}, the generalized algebraic connectivity
$\kappa$ satisfies
\[
\kappa \le \frac{1}{n-1}\sum_{j\to i\in E} a_{ij}.
\]
\end{proposition}
\begin{proof}
Let $0=\lambda_1(M)<\lambda_2(M)=\kappa\le \cdots \le \lambda_n(M)$
be the eigenvalues of
\[
M=\Xi^{-1/2}\frac{\Xi L+L^\top \Xi}{2}\Xi^{-1/2}.
\]
Since $M$ is symmetric positive semidefinite, we have
\[
(n-1)\kappa \le \sum_{r=2}^n \lambda_r(M)=\operatorname{tr}(M).
\]
Moreover, by the cyclic
invariance of the trace, 
\[
\operatorname{tr}(M)
=
\frac12\operatorname{tr}(\Xi L\Xi^{-1})
+
\frac12\operatorname{tr}(L^\top\Xi\Xi^{-1})=\operatorname{tr}(L).
\]
Therefore,
\[
(n-1)\kappa \le \operatorname{tr}(L)=
\sum_{i=1}^n L_{ii}
=
\sum_{i=1}^n\sum_{j\ne i}a_{ij}
=
\sum_{j\to i\in E}a_{ij}.
\]
This completes the proof.
\end{proof}

\ifCLASSOPTIONcompsoc
  \section*{Acknowledgments}
\else
  \section*{Acknowledgment}
\fi

The authors gratefully acknowledge support from the Anhui Center for Applied Mathematics and the Key Laboratory of the Ministry of Education for Mathematical Foundations and Applications of Digital Technology.

The implementation of the proposed methods is publicly available at:
\url{https://github.com/xinyu970905-cloud/SpectralSensitivity}.

\ifCLASSOPTIONcaptionsoff
  \newpage
\fi

\bibliographystyle{IEEEtran}
\bibliography{IEEEabrv,Laplacian}

\end{document}